\documentclass[a4paper,12pt, reqno]{amsart}
\usepackage{amsfonts, color}
\usepackage[textwidth=460pt]{geometry}
\usepackage{amstext, amsthm, amssymb, amsmath}
\usepackage{dsfont}
\usepackage{mathtools}
\usepackage{mathrsfs}
\usepackage{graphicx}
\usepackage{color}
\usepackage{todonotes}
\usepackage[linesnumbered,lined,ruled,commentsnumbered]{algorithm2e}

\usepackage{hyperref}

\usepackage{cleveref}
\hypersetup{colorlinks = true, citecolor = blue}

\usepackage{tcolorbox}
\usepackage{booktabs}

\theoremstyle{plain}
\begingroup
\newtheorem{theorem}{Theorem}[section]
\newtheorem{lemma}[theorem]{Lemma}
\newtheorem{proposition}[theorem]{Proposition}
\newtheorem{corollary}[theorem]{Corollary}
\endgroup

\theoremstyle{definition}
\begingroup

\endgroup

\theoremstyle{remark}
\newtheorem{remark}{Remark}

\numberwithin{equation}{section}

\newcommand{\J}{\mathcal{J}}
\newcommand{\R}{\mathbb{R}}

\newcommand{\U}{\mathcal{U}}

\newcommand{\e}{\varepsilon}
\newcommand{\weak}{\rightharpoonup}

\newcommand{\de}{\,\mathrm{d}}
\newcommand{\FundingLogos}{%
  \raisebox{0pt}{\includegraphics[height=1.5cm]{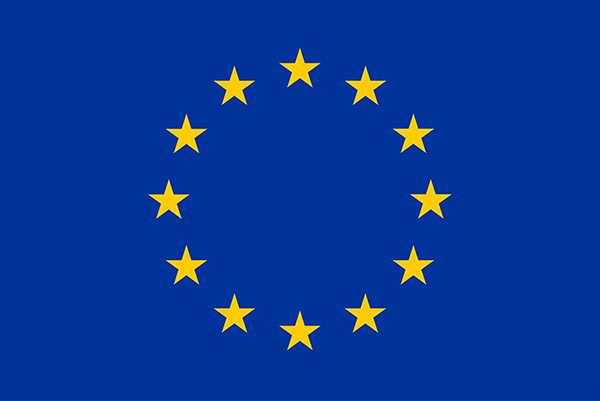}}%
  \hspace{1em}%
  \raisebox{0pt}{\includegraphics[height=1.5cm]{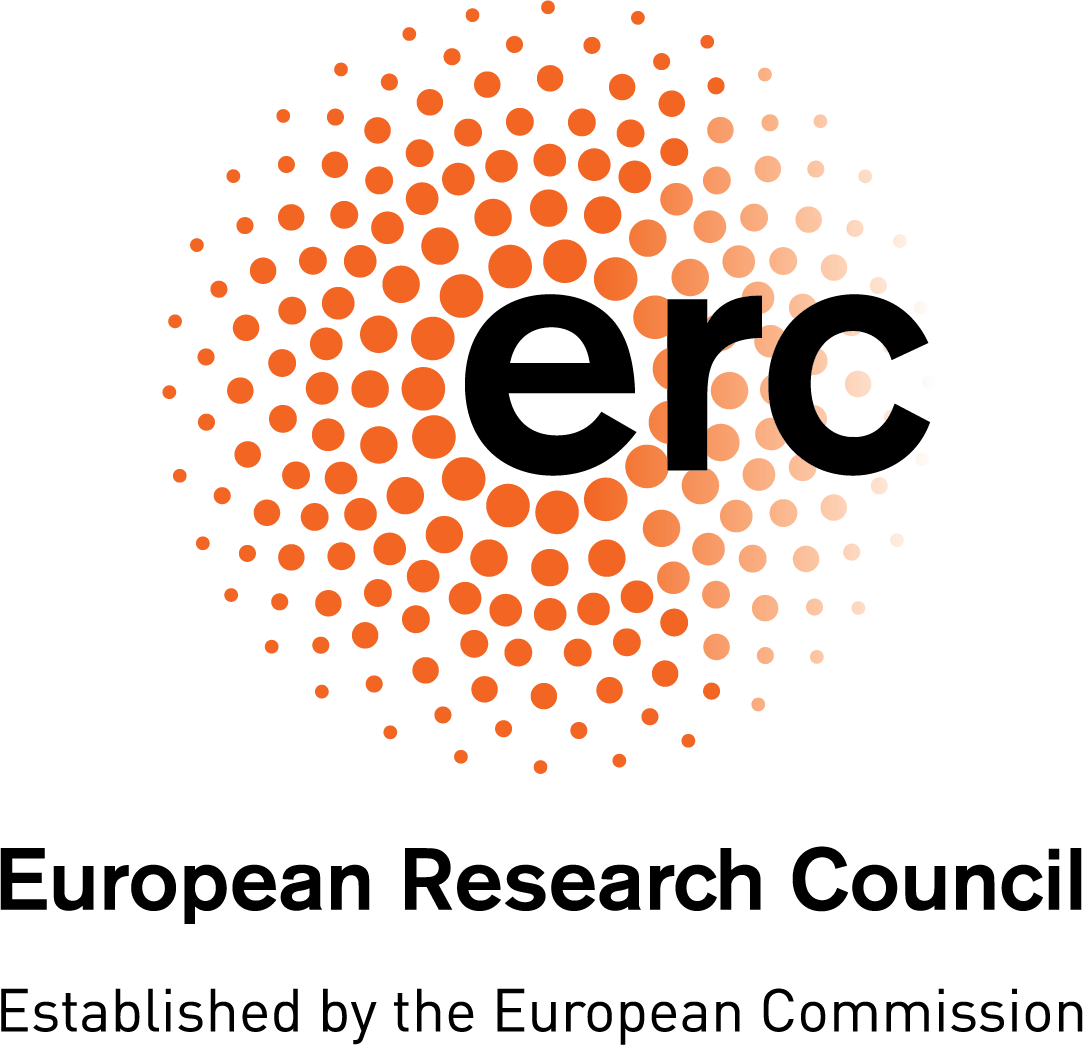}}%
}

\crefname{theorem}{Theorem}{Theorems}
\crefname{lemma}{Lemma}{Lemmas}
\crefname{corollary}{Corollary}{Corollaries}
\crefname{section}{Section}{Sections}
\crefname{subsection}{Subsection}{Subsections}
\crefname{proposition}{Proposition}{Proposition}
\crefname{defn}{Definition}{Definitions}
\crefname{appendix}{Appendix}{Appendices}
\crefname{remark}{Remark}{Remarks}
\crefname{table}{Table}{Tables}

\usepackage[foot]{amsaddr}

\title[Ensemble Optimal Control for drug resistance in cancer therapies]{Ensemble Optimal Control for managing drug resistance in cancer therapies}

\author[A. Scagliotti]{Alessandro Scagliotti$^{1,2}$}
\email{scag@ma.tum.de}
\address{$^1$CIT School, Technical University of Munich, Garching bei M\"unchen, Germany.}
\address{$^2$Munich Center for Machine Learning (MCML), Munich, Germany.}

\author[F. Scagliotti]{Federico Scagliotti$^{3,4}$}
\email{federico.scagliotti@icsmaugeri.it}
\address{$^3$Medical Oncology Unit, Istituti Clinici Scientifici Maugeri IRCCS,
Pavia, Italy.}
\address{$^4$Department of Internal Medicine and Medical Therapeutics, University of
Pavia, Pavia, Italy.}

\author[L.D. Locati]{Laura Deborah Locati$^{3,4}$}
\email{lauradeborah.locati@unipv.it}

\author[F. Sottotetti]{Federico Sottotetti$^{3,4}$}
\email{federico.sottotetti@icsmaugeri.it}

\date{\today}

\begin{document}

\maketitle

\begin{abstract}
    In this paper, we explore the application of ensemble optimal control to derive enhanced strategies for pharmacological cancer treatment, and we tackle the problem of the long-term management of the disease, i.e., when the complete eradication of the tumor is not achievable.
In particular, we focus on moving beyond the classical clinical approach of giving the patient the maximal tolerated drug dose (MTD), which does not properly exploit the fight among sensitive and resistant cells for the available resources.
Here, we employ a Lotka-Volterra model to describe the competing subpopulations, and we enclose this system within the ensemble control framework.
In the first part, we establish general results suitable for application to various cancers.
Then, we carry out numerical simulations in the setting of prostate cancer treated with androgen deprivation therapy, yielding a computed policy that is reminiscent of the medical `active surveillance' paradigm.
Finally, inspired by the numerical evidence, we propose a variant of the celebrated adaptive therapy (AT), which we call `Off-On' AT.

    \subsection*{Keywords:} Ensemble optimal control, Cancer modeling, Drug resistance management, Gradient-based optimization.

    \subsection*{Mathematics Subject Classification:}
49M25, 49M05, 92C50, 49J45.

\end{abstract}

\section{Introduction} \label{sec:intro}

In this paper, we consider a simple differential model for drug resistance in pharmacological cancer treatments, and we address the uncertainty that affects the dynamics and the Cauchy datum using tools of ensemble control of ODEs.

We recall that{, given a time horizon $T>0$,} an \emph{ensemble of control systems} {in $\R^n$} is a family of controlled ODEs of the form 
\begin{equation} \label{eq:intro_gener_ens}
    \begin{cases}
        \dot x^\theta(t) = G^\theta\big(t, x^\theta(t), u(t) \big) & \mbox{a.e. in } [0,T], \\
        x^\theta(0) = x_0^\theta,
    \end{cases}
\end{equation}
where $u\colon [0,T] \to \R^m$ is the {$m$-dimensional} control {input, $[0,T]\ni t\mapsto x^\theta(t) \in \R^n$ is the corresponding trajectory}, and where the dynamics $G^\theta\colon [0,T]\times \R^n\times \R^m\to\R^n$ and the initial condition $x_0^\theta$ depend continuously on the (unknown) {$k$-dimensional} parameter $\theta$, which varies in a compact set $\Theta\subset \R^k$. 
In this framework, we aim to find a common policy $t \mapsto u(t)$ for \emph{simultaneously} driving every system of the parametrized family \eqref{eq:intro_gener_ens}. 
This is typically the case when the parameters that appear in the dynamics are subject to statistical errors, and we seek to compute a control that works in every situation (see \cite{RuLi12}). 
Problems involving ensembles are a timely source of interest for researchers working in Mathematical Control. Indeed, on the one hand, among the recent theoretical contributions in the field, we report \cite{aronna2025average,BK19,BK24} for \emph{averaged} optimal control problems, and \cite{AB25,Scag25} for the \emph{minimax} optimization.
Moreover, the Hamilton-Jacobi-Bellman equation related to ensemble control has been considered in \cite{aronna2024dynamic}.
On the other hand, owing to its flexibility, this framework has found several applications, e.g.~in quantum control \cite{AuBoSi18,ChiGau18,RABS}, and in the mathematical modeling of Deep Learning (see \cite{RuZu23,ASZu,ABR24,Cipriani2025}) and of Reinforcement Learning (see \cite{MP18,PPF21}), to mention a few.
To the best of our knowledge, here we employ for the first time the viewpoint of ensemble control for tuning drug dosage in cancer therapy. 
Nonetheless, the interplay between optimal control and tumor modeling has a long and glorious story (see the textbook \cite{schattler2015optimal} for an introduction and for a historical overview).
Among recent contributions, we refer the reader to \cite{cunningham2018optimal,ledzewicz2024optimal}, and to \cite{edduweh2024liouville} for a Liouville equation for prostate cancer. Moreover, we mention 
\cite{aguade2024modeling} for a generalized Lotka-Volterra model for the ecological variety of tumors environment, and \cite{morselli2023agent,morselli2024hybrid} for models on oncolytic viruses.

{In the first part of the paper (\cref{sec:theory}), we develop in a general setting a framework for ensemble optimal control for cancer therapy.
The only major requirement is that the differential model must be affine in the control variable, i.e., it should be an $n$-dimensional system of ODEs of the form
\begin{equation} \label{eq:intro_affine_syst}
    \begin{cases}
        \dot Z^\theta(t) = F_0^\theta\big(Z^\theta(t)\big) + F_1^\theta\big(Z^\theta(t)\big)w(t), \\
        Z^\theta(0) = Z_0^\theta,
    \end{cases}
\end{equation}
where $t\mapsto w(t) \in \mathbb R$ is the control that acts on the system, $\theta\in\Theta$ collects the uncertain parameters, and $F_0^\theta, F_1 ^\theta \colon \R^n\to \R^n$ are, respectively, the uncontrolled and the controlled part of the dynamics. 
This setting encloses rather sophisticated multi-species models with secondary resistance mechanisms, like, e.g., 
\begin{equation} \label{eq:intro_soph_model}
\begin{split}
    &\begin{cases}
        \dot X_i^\theta(t) = r_i \left( 1- \frac1K \sum_{j=1}^n X_j^\theta(t) \right)  \left(  1-2d_i^I \frac{D(t)}{D_{\max}}\right) X_i^\theta(t) - d_i^T X_i^\theta(t) \\
        \qquad \qquad + \sum_{j=1}^n\left( A_{i,j} + A^I_{i,j}\frac{D(t)}{D_{\max}} \right) X_j^\theta(t) \qquad \mbox{for }i=1,\ldots,n, \\
        X^\theta_i(0) = f_i N_0 \qquad\qquad\qquad\qquad\qquad\qquad\quad \quad \, \mbox{for }i=1,\ldots,n,
    \end{cases}\\
& \quad N^\theta(t) = \sum_{i=1}^n X_i^\theta(t),\\
& \quad \theta = \left( (r_i)_{i}, K, (d^I_i)_{i}, (d^T_i)_{i},  (A_{i,j})_{i,j} , (A_{i,j}^I)_{i,j}, (f_i)_i, N_0 \right) \in \Theta,
    \end{split}
\end{equation}
where $t\mapsto D(t)$ is the control and represents the dose of the drug given to a patient, and $t\mapsto X_i^\theta(t)$ is the function describing the size of the $i$-th cancer sub-population as the time $t$ varies (see also \cref{tab:parameters}).
\begin{table}[b]
\centering
\begin{tabular}{lll}
\toprule
\textbf{Parameter} & \textbf{Meaning} & \textbf{Units} \\
\midrule
$X_i(t)$ & Size of subpopulation $i$ & cells \\
$N(t)$ & Total population $\sum_{i=1}^n X_i(t)$ & cells \\
$r_i$ & Proliferation rate of subpopulation $i$ & time$^{-1}$ \\
$K$ & Carrying capacity & cells \\
$d_i^T$ & Turnover (natural death) rate & time$^{-1}$ \\
$d_i^I$ & Drug sensitivity coefficient of subpopulation $i$ & dimensionless \\
$A_{i,j} \ i\neq j$ & Spontaneous transition rate $j \to i$ & time$^{-1}$ \\
$A_{i,i} \leq 0$ & Spontaneous `evolution rate' of subpopulation $i$ & time$^{-1}$ \\
$A^I_{i,j} \ i\neq j$ & Drug-induced transition rate $j \to i$ & time$^{-1}$ \\
$A_{i,i}^I \leq 0$ & Drug-induced `evolution rate' of subpopulation $i$ & time$^{-1}$ \\
$D(t)$ & Drug dosage (control) & dosage \\
$D_{\max}$ & Maximum drug dosage & dosage \\
$f_i$ & Initial fraction of subpopulation $i$ & dimensionless \\
$N_0$ & Initial total population & cells \\
\bottomrule
\end{tabular}
\caption{Model parameters and variables for the system in \cref{eq:intro_soph_model}.}
\label{tab:parameters}
\end{table}
In fact, we can consider even more refined models than the one in \eqref{eq:intro_soph_model}, where the drug concentration around the cancer cells does coincide with the drug dose.
In a more realistic description, concentration $t\mapsto C(t)$ and dose $t\mapsto D(t)$ can be related via a differential equation, such as, e.g.,
\begin{equation*}
    \dot C(t) = -k^C C(t) + k^I \frac{D(t)}{D_{\max}} .
\end{equation*}
In this case, we set $Z\coloneqq (C,X_1,\ldots,X_n)$ and we modify accordingly \eqref{eq:intro_soph_model} so that the induced death rate and the inter-species evolution is triggered by $C(t)$ in place of $D(t)$. Moreover, a direct computation shows that $t\mapsto Z(t)$ solves as well a control-affine system of the form \eqref{eq:intro_affine_syst}, where $t\mapsto D(t)$ is once again the control input. We insist on the fact that the affine dependence in the control variable is crucial for our analysis---in particular, for the results contained in \cref{thm:exist_Gamma}.
Once a suitable control-affine model has been selected, the ensemble optimal control approach consists in finding a drug administration policy $t\mapsto D(t)$ that takes into account the uncertainty that affects \emph{(some of)} the model parameters.}
The goal {of this paper} is to propose strategies that exhibit enhanced performances in the long-term disease management. In particular, we focus on moving beyond the classical clinical approach of giving to the patient the maximal tolerated dose (MTD), which consists in setting ${D}_{\mathrm{MTD}}(t) = {D_{\max}}$ for every $t>0$ in \cref{eq:intro_soph_model}.
When sensitive and resistant clones compete for the available resources---as in solid cancers\footnote{{In hematological diseases, cancer cells typically lack strong spatial constraints and localized resource limitation, which are key drivers of ecological competition in solid tumors. However, the theoretical results in \cref{sec:theory} are general and also apply to hematological malignancies, provided the underlying dynamical model can be formulated as a control-affine system of ODEs.}}---, MTD has already been shown to be sub-optimal. Indeed, it leads to the extinction of the sensitive cells and to the emerging of a resistant population \cite{moff2012evolutionary}.  
Here, we aim to obtain a policy of drug dosage by solving an ensemble optimal control problem over {a fixed} time horizon $[0,T]${, where $T$ is pre-determined (see \cref{rmk:choice_T} for the choice in our simulations)}.
To do this, the key-steps are the definition of a compact set $\Theta$ where the {uncertain} parameters $\theta$ vary, and the introduction of a probability measure $\mu \in \mathcal{P}(\Theta)$ that describes our knowledge about the distribution of $\theta$. {For instance, one may consider repeated measurements of tumor burden in a cohort of patients over the course of therapy, where treatment interruptions (vacations) are employed. Model parameters for each patient could then be estimated via standard parameter identification procedures (see, e.g., \cite{Nelles2020}). The resulting collection of estimates $\theta_1,\ldots,\theta_k$ may be interpreted as independent samples from the underlying distribution, yielding the definition of an empirical probability measure $\mu\coloneqq \frac1k \sum_{j=1}^k \delta_{\theta_j}$,  which can be used to study new or existing treatment schedules through ensemble optimal control. Hence, in this framework}, the dosage strategy is computed through the minimization of a functional $\J\colon \U\to\R$ of the form
\begin{equation} \label{eq:intro_funct}
    \J({D}) \coloneqq  
    \int_\Theta  \int_0^T \ell\left( {N_D^\theta(t)} \right)  
    \de {t} \de\mu(\theta),
\end{equation}
where $\U\coloneqq \{{D} \in L^2([0,T],\R): 0\leq {D(t)}\leq {D_{\max}} \, \mbox{ for a.e. } {t}\}$, ${N_D^\theta(t)}$ is derived from \cref{eq:intro_soph_model}, and $\ell\colon \R\to\R$ is a proper penalization cost on the tumor size. {For the problem of minimizing \eqref{eq:intro_funct}, we provide a $\Gamma$-convergence result (see \cref{thm:exist_Gamma}), and we provide the explicit expression for the projected gradient field (\cref{prop:grad_rep,prop:g_descent}), moving beyond \cite{scag23gf,Scag23}, where the control variable was unconstrained.}
{Finally, we observe that,} roughly speaking, minimizing \eqref{eq:intro_funct} means that we look for a control that `does \emph{on average} a good job on the elements of the ensemble'.
In this regard, we point out that, even though a \emph{minimax} ensemble control formulation is a viable option (see \cite{Scag25}), in the application that we are considering here it is rather distant from the clinical practice. Indeed, the focus on the improvement of the least favorable cases usually comes at the expenses of a performance degradation on the most likely ones.

{In the second part of the paper}, we study a {simplified} system that describes the {dynamics} of the two sub\-populations of a tumor (sensitive and resistant) when the patient is undergoing a pharmacological treatment. {Following \cite{cunningham2018optimal,turnover2021}, in the present analysis we assume that the resistant sub\-population consists of descendants of pre-existing clones, and we do not include a mechanism of acquired (secondary) resistance.}
Namely, adopting the same notations as in \cite{turnover2021}, we render the competition for the limited resources through a Lotka-Volterra system, and we address the following ODE in $\R^2$:
\begin{equation}\label{eq:intro_model}
\begin{split}
    &\begin{cases}
        \dot S(t) = r_S\left( 1 - \frac{S(t)+R(t)}{K} \right)\left( 1 - 2d_D\frac{D(t)}{D_{\max}} \right)S(t) - d_T S(t), \\
    \dot R(t) = r_R\left( 1 - \frac{S(t)+R(t)}{K} \right) R(t) - d_T R(t),
    \end{cases}
    \\ 
    & \quad N(t) \coloneqq S(t) + R(t),
\end{split}
\end{equation}
where the time-dependent functions $t\mapsto S(t)$ and $t\mapsto R(t)$ denote, respectively, the amount of cells of the tumor population at the instant $t$ that are sensitive and resistant to a certain drug, while $t\mapsto N(t)$ accounts for the total population {(see also \cref{tab:simplified_parameters})}.
In \cref{eq:intro_model} the control is $t\mapsto D(t)$, which takes value in $[0,D_{\max}]$ and represents the drug concentration in the tumor environment. In our model, we assume that the fraction of cells killed by the drug is linearly proportional to the given dose. This classical hypothesis, formulated in \cite{NortonSimon77}, has been widely adopted in the literature (see, e.g., \cite{schattler2015optimal,turnover2021}).
\begin{table}[b]
\centering
\begin{tabular}{lll}
\toprule
\textbf{Parameter} & \textbf{Meaning} & \textbf{Units} \\
\midrule
$S(t)$ & Size of sensitive subpopulation & cells \\
$R(t)$ & Size of resistant subpopulation & cells \\
$N(t)$ & Total population $S(t)+R(t)$ & cells \\
$r_S$ & Proliferation rate of sensitive cells & time$^{-1}$ \\
$r_R$ & Proliferation rate of resistant cells & time$^{-1}$ \\
$K$ & Carrying capacity & cells \\
$d_T$ & Turnover (natural death) rate & time$^{-1}$ \\
$d_D$ & Drug sensitivity coefficient (sensitive cells) & dimensionless \\
$D(t)$ & Drug dosage (control) & dosage \\
$D_{\max}$ & Maximum drug dosage & dosage \\
\bottomrule
\end{tabular}
\caption{Model parameters and variables for the simplified two-population model in \cref{eq:intro_model}.}
\label{tab:simplified_parameters}
\end{table}
In \cref{sec:experiments} we take advantage of the parameters estimates reported in \cite{turnover2021} and we formulate an ensemble optimal control problem for \emph{prostate cancer} (non-metastatic and castration sensitive, m0CSPC) \emph{treated with androgen deprivation therapy (ADT)}. 
We compare the performances using the `time to progression' of the disease (TTP), and we adopt the Adaptive Therapy (AT) proposed in \cite{moff2009adaptive,moff2017zhang} and MTD as benchmark strategies.
When the integral cost $\ell$ is linear, the resulting optimal strategy turns out to behave closely to the MTD protocol. However, when $\ell$ has a hyperbolic profile, the computed policy suggests a delayed starting of the therapy, and on average it outperforms MTD and AT in terms of `time to progression' (see \cref{table:MTD_AT1,table:Sum_Hyp}). Moreover, this strategy is reminiscent of medical paradigm known as `active surveillance'   \cite{preisser2024european,karim2025early}. 
{While this feature may be of conceptual interest, a significant limitation of the computed policy is that, during the initial phase, it permits substantial tumor growth, which may limit its practical applicability.}
Finally, to amend this point, we propose a variant of the AT that we call `Off-On' Adaptive Therapy, which combines the `active surveillance' paradigm (i.e., delayed starting of the therapy) with adaptive periods of treatment vacation \cite{moff2009adaptive,moff2017zhang}, showing promising results (see \cref{table:Off-On}).

\section{Ensemble optimal control formulation: theoretical framework}
\label{sec:theory}

In this section, we perform a general analysis for the ensemble control problems formulated in \cref{eq:intro_soph_model}.
First, we show how to enclose \cref{eq:intro_soph_model} within the theoretical framework of ensembles of control-affine systems tackled in \cite{Scag23,Scag25}.
Then, we define the ensemble optimal control problem, and we show that it admits minimizers and that a $\Gamma$-convergence result holds.
Finally, we study the gradient of the functional involved in the optimization problem.

\subsection{Definition of the dynamics} \label{subsec:theoretical_dyn}
We recall that an ensemble of control-affine systems has the form $\dot X^\theta = \tilde F_0^\theta(X^\theta) + \tilde F_1^\theta(X^\theta)u$.
In order to identify the drift $\tilde F_0^\theta$ and the controlled vector field $\tilde F_1^\theta$, we set {$\theta \coloneqq \left( (r_i)_{i}, K, (d^I_i)_{i}, (d^T_i)_{i},  (A_{i,j})_{i,j} , (A_{i,j}^I)_{i,j}, (f_i)_i, N_0 \right) \in \Theta$}, and we rearrange \cref{eq:intro_soph_model} as follows:
\begin{equation} \label{eq:aff_ctrl_sys}
    \begin{split}
    &\frac{\mathrm{d}}{\mathrm{d}t} X_D^\theta(t)= \frac{\mathrm{d}}{\mathrm{d}t}
    \left( X_{D\ i}^\theta(t)
    \right)_{i=1,\ldots,n}
    = \tilde F_0^\theta\left(X_D^\theta(t)\right)
        +
        \tilde F_1^\theta\left(X_D^\theta(t)\right) D(t)
    \\
    &= \left( { \textstyle
    r_i \left( 1- \frac1K \sum_{j=1}^n X_j^\theta(t) \right)   X_i^\theta(t) - d_i^T X_i^\theta(t)  + \sum_{j=1}^n  A_{i,j} X_j^\theta(t) }
    \right)_{i=1,\ldots,n}\\
    &\quad + 
    \left( { \textstyle
    -r_i \left( 1- \frac1K \sum_{j=1}^n X_j^\theta(t) \right)   \frac{2d_i^I}{D_{\max}} X_i^\theta(t)   + \sum_{j=1}^n\frac{A^I_{i,j}}{D_{\max}} X_j^\theta(t) }
    \right)_{i=1,\ldots,n} D(t),
    \end{split}
\end{equation}
with the initial condition
\begin{equation} \label{eq:Cau_aff_ctrl}
    X^\theta_{D\ i}(0) = f_i N_0 \qquad i=1,\ldots,n,
\end{equation}
where {$N_0\in[0,K]$} denotes the initial size of the tumor.
{We briefly discuss below the role of the parameters appearing in \cref{eq:aff_ctrl_sys}:
\begin{itemize}
    \item[-]$r_i>0$ denotes the reproduction rate of the $i$-th sub\-population, for $i=1,\ldots,n$; 
    \item[-]$K, N_0$ denote, respectively, the carrying capacity of the system and the total initial population of cancer cells;
    \item[-]$d_i^I\geq 0$ denotes the drug sensitivity coefficient of the $i$-th  sub\-population, for $i=1,\ldots,n$. Without loss of generality, we can assume that $d_1^I\geq d_2^I\geq\ldots\geq d_n^I\geq0$;
    \item[-]$d_i^T\geq 0$ denotes the turnover rate of the $i$-th  sub\-population, for $i=1,\ldots,n$;
    \item[-]$f_i\geq 0$ denotes the fraction of the initial total population $N_0$ that belongs to the $i$-th sub\-population, for $i=1,\ldots,n$;
    \item[-]$A_{i,j}\geq 0$ denotes the `spontaneous' rate of the evolution (in the Darwinian sense) of the $j$-th sub\-population into the $i$-th sub\-population, for $i,j=1,\ldots,n$ and $i\neq j$;
    \item[-]$A_{i,j}^I\geq 0$ denotes the drug-induced rate of the evolution (in the Darwinian sense) of the $j$-th sub\-population into the $i$-th sub\-population, for $i,j=1,\ldots,n$ and $i\neq j$;
    \item[-]$A_{i,i}\leq 0$ and $A^I_{i,i}\leq 0$ denote, respectively, the `spontaneous' and drug-induced rate of the $i$-th sub\-population with which it evolves into some other species (`evolution rate').
\end{itemize}
We observe that the condition $d_i^I=0$ implies that the $i$-th population is completely insensitive to the drug.
We insist on the fact that, in order enforce that the inter-species evolution does not increase the total population, we require the following constraints on $(A_{i,j})_{i,j} , (A_{i,j}^I)_{i,j}$:
\begin{equation} \label{eq:pop_conservation}
     A_{i,i} + \sum_{j\neq i}A_{j,i}=0, \qquad  A_{i,i}^I + \sum_{j\neq i}A_{j,i}^I=0,
\end{equation}
together with the sign constraints $A_{i,i}\leq 0$ and $A^I_{i,i}\leq 0$, and $A_{i,j}\geq 0$ and $A_{i,j}^I\geq 0$ for $i\neq j$.
Moreover, we assume that $K\in [K_{\min},K_{\max}]$ and $N_0\in [0,K]$, and that $\sum_{i=1}^n f_i=1$.
We observe that
\begin{equation} \label{eq:simplex_start_theta}
    \big( X_{D\, i}(0)^\theta \big)_{i=1,\ldots,n} \in \Delta^\theta \coloneqq \left\{ X \in \R^n : X_i \geq 0 \ \forall i, \, 0\leq \sum_{i=1}^n X_i \leq K \right\}
\end{equation}
for every $\theta \in \Theta$, and we introduce \begin{equation} \label{eq:simplex_start}
     \Delta \coloneqq \left\{ X \in \R^n : X_i \geq 0 \ \forall i, \, 0\leq \sum_{i=1}^n X_i \leq K_{\max} \right\},
\end{equation} 
which satisfies $\Delta\supset \Delta^\theta$ for every $\theta\in\Theta$.}
Given $T>0$, we consider as the space of admissible controls
\begin{equation}\label{eq:def_adm_ctrls}
    \U\coloneqq \{ {D} \in L^2([0,T],\R): 0\leq {D(t)\leq D_{\max} } \, \mbox{ for a.e. } t \}.
\end{equation}
We notice that $\U$ is a convex subset of $L^2([0,T],\R)$, which we equip with the usual Hilbert space structure.
We first show that the simplex $\Delta$ defined above is invariant for the dynamics that we are considering.
\begin{lemma}  \label{lem:domain_invariance}
    Let us consider ${D}\in \U$ defined as in \cref{eq:def_adm_ctrls}, and the simplex $\Delta \in \R^{n}$ introduced in \cref{eq:simplex_start}.
    For every $\theta \in \Theta$, if ${X_D^\theta(0)}\in \Delta^{\theta}$, then ${X_D^\theta(t)} \in \Delta^{\theta}$ {(and, in particular, $X_D^\theta(t) \in \Delta$)} for every $t \in [0,T]$.
\end{lemma}
\begin{proof}
    See Appendix~\ref{sec:appendix_theory}.
\end{proof}

Unfortunately, the vector fields $\tilde F_0^\theta,\tilde F_1^\theta \colon \R^{n} \to \R^{n}$ involved in the differential model \eqref{eq:aff_ctrl_sys} do not satisfy the usual working assumptions, which require the fields to be globally Lipschitz continuous and to have a sub-linear growth (see, e.g., \cite[Hypothesis~2.1]{Scag25}).
However, \cref{lem:domain_invariance} enables us to circumvent this issue. Indeed, we can define the vector fields $F_0^\theta ,F_1^\theta\colon \R^{n} \to \R^{n}$ as follows:
\begin{equation} \label{eq:def_trunc_fields}
    F_0^\theta({X}) \coloneqq \rho({X})\tilde F_0^\theta ({X}) 
    \quad \mbox{and} \quad
    F_1^\theta ({X}) \coloneqq \rho({X}) \tilde F_1^\theta ({X}) 
\end{equation}
for every ${X} \in \R^{n}$ and for every $\theta\in \Theta$, where $\rho\colon \R^{n}\to [0,1]$ is a smooth cut-off function such that $\rho\equiv 1$ in $B_{{2K_{\max}}}(0)$ and $\mathrm{supp}{(\rho)}\subset B_{{3K_{\max}}}(0)$. 
On the one hand, considering $F_0^\theta ,F_1^\theta $ in place of $\tilde F_0^\theta ,\tilde F_1^\theta$ does not alter the trajectories that are of interest for our model, owing to \cref{lem:domain_invariance} and since $\Delta \subset B_{{2K_{\max}}}(0)$. 
On the other hand, $F_1^\theta , F_2^\theta$ fall perfectly within the classical theoretical framework of ensembles of control-affine systems. 

\subsection{Ensemble optimal control problem and related functional}
We shall propose a therapy schedule resulting {in the} (approximate) solution of an ensemble optimal control problem. 
To this end, we assume that we are given a probability measure $\mu\in\mathcal{P}(\Theta)$ on the space of the unknown parameters $\theta$ of the control system in \cref{eq:aff_ctrl_sys,eq:Cau_aff_ctrl}. Throughout the paper, the space $\Theta\subset \R^{m}$ is assumed to be compact.
Recalling the definition of $\U$ in \cref{eq:def_adm_ctrls}, we consider the ensemble optimal control problem associated to the functional $\J\colon \U\to \R$ defined as 
\begin{equation} \label{eq:def_ens_funct}
    \J({D}) \coloneqq  \int_\Theta  \int_0^T \ell\left({N_D}^\theta({t} )\right)  
    \de {t} \de\mu(\theta),
\end{equation}
where {$N_D^\theta(t)=\sum_{i=1}^nX_{D\ i}^\theta(t)$, with $X_{D\, 1}^\theta, \ldots, X_{D\, n}^\theta$} solutions of \cref{eq:intro_model_nondim} corresponding to the admissible control ${D}\in\U$ and to the parameters $\theta\in\Theta$, and where $\ell\colon \R \to\R$ is a cost function of class $C^2$.

\begin{remark} \label{rmk:struct_cost}
    The function that we integrate in \cref{eq:def_ens_funct} depends \emph{on the total cancer population ${N_D^\theta(t)}$ at every instant $t\in [0,T]$}. 
    This choice has a precise interpretation in the model. Indeed, in a real-world scenario, it is not realistic to have access separately to the sizes of the {$n$ different sub\-populations}.
\end{remark}

\begin{remark}
    The functional $\J$ introduced in \cref{eq:def_ens_funct} designs an \emph{averaged} ensemble optimal control problem, as we saturate the dependence on the parameter $\theta$ by averaging with respect to the probability measure $\mu$. 
    An alternative paradigm consists in optimizing with respect to the worst-case scenario. Namely, this would result in addressing the minimax problem induced by the functional
    \begin{equation} \label{eq:def_Funct_minimax}
        \J_{\max} \coloneqq 
        \max_{\theta \in \Theta} \left(
        \int_0^T \ell\left({N_D}^\theta({t} )\right)  
    \de {t} \right),
    \end{equation}
    recently studied in \cite{Scag25}.
    However, in the medical application that we are considering in the present paper, a minimax formulation can lead to over-conservative strategies.
    Indeed, it is more natural---and closer to the clinical practice---to design treatment strategies that are effective for the majority of the patients, rather than trying to achieve the best possible outcome on the worst-cases. This is due to the fact that the improvement on the least favorable systems of the ensemble usually comes at the expenses of a performance degradation on the most likely ones (see \cite[Section~6]{Scag25}).
    {Nevertheless, for comparison, in \cref{sec:experiments} we include numerical simulations addressing the minimax ensemble control problem.}
\end{remark}

Below, we show that the ensemble optimal control problem consisting in minimizing $\J$ admits solution. Moreover, using the tools of $\Gamma$-convergence, we establish a `stability result' for the minimizers of $\J$ with respect to perturbations in $\mu$. 
Given a sequence $(\mu_{k})_{{k}\geq1}\in\mathcal{P}(\Theta)$, we write $\mu_{k}\weak^* \mu$ as ${k}\to\infty$ to denote the weak-$*$ convergence of probability measures, i.e., 
\begin{equation*}
    \lim_{{k}\to\infty} \int_\Theta \phi(\theta)\de \mu_{k}(\theta) = 
    \int_\Theta \phi(\theta)\de \mu(\theta)
\end{equation*}
for every $\phi:\Theta\to\R$ bounded and continuous.

\begin{theorem} \label{thm:exist_Gamma}
    The functional $\J\colon \U\to \R$ defined in \cref{eq:def_ens_funct} admits minimizer in $\U$.\\
    Moreover, given a sequence $(\mu_{k})_{{k}\geq1}\in\mathcal{P}(\Theta)$ such that $\mu_{k}\weak^* \mu$ as ${k}\to\infty$, if we define 
    \begin{equation*}
        \J_{k}({D}) \coloneqq  \int_\Theta  \int_0^T \ell\left({ N_D^\theta(t )}\right)  
    \de t \de\mu_{k}(\theta)
    \end{equation*}
    for every ${D}\in\U$ and for every ${k}\geq1$, and if we denote with ${D}^\star_{k}$ a minimizer of $\J_{k}$ for every ${k}\geq1$, then we have $\min_\U \J = \lim_{{k}\to\infty}\J_{k}({D}^\star_{k})$, and every $L^2$-weak limiting point of $({D}^\star_{k})_{{k}\geq1}$ is a minimizer for $\J$. 
\end{theorem}
\begin{proof}
    The proof of the first part follows using the direct method of the Calculus of Variations, considering the weak topology of $L^2$. More precisely, the lower semi-continuity of $\J$ is contained as a particular case of the proof of \cite[Theorem~3.2]{Scag23}. 
    For the weak coercivity, we recall that the space of admissible controls $\U$ is strongly closed and convex, hence it is weakly closed. Moreover, it is contained in the ball of $L^2$ centered at the origin and with radius $\sqrt{T}$. Therefore, the whole domain $\U$ of $\J$ is weakly compact, and $\J$ admits minimizer.
    The second part part of the statement descends from the fact that the sequence of functionals $(\J_{k})_{{k}\geq1}$ is $\Gamma$-convergent to the functional $\J$ with respect to the $L^2$-weak topology (see \cite[Theorem~4.6]{Scag23}). Hence, observing that the functionals $(\J_{k})_{{k}\geq1}$ are defined on the same weakly compact domain $\U$, the thesis follows by using \cite[Corollary~4.8]{Scag23}.
\end{proof}

\begin{remark} \label{rmk:G_conv}
    The second part of \cref{thm:exist_Gamma} plays a pivotal role in the applications. Indeed, each evaluation of the functional $\J$ requires the computation of the trajectories $t\mapsto {X_D}^\theta(t)$ solving \cref{eq:aff_ctrl_sys,eq:Cau_aff_ctrl} for every $\theta\in \mathrm{supp}(\mu)$. 
    When the support of the measure $\mu\in\mathcal{P}(\Theta)$ contains infinitely many elements, this turns out to be completely impractical.
    In the case we approximate $\mu$ with a discrete measure $\mu_{k}$ with finite support, \cref{thm:exist_Gamma} ensures that any minimizer of $\J_{k}$ is a good competitor as well for the original objective functional $\J$. Finally, we observe that the construction of discrete approximating measures is an active research field (see, e.g., \cite{merigot2021non,auricchio2024facility,auricchio2024k,auricchio2024extended})
\end{remark}

{ 
\begin{remark} \label{rmk:bang_bang}
    We recall that the Pontryagin Maximum Principle (PMP) provides necessary conditions for optimal controls (see, e.g., \cite[Section~6]{Bressan_Piccoli} for a general introduction).
    In the framework of averaged ensemble optimal control, the PMP has been studied in \cite{BK19}, and it provides insights on the structure of minimizers.
    More precisely, in our setting where the cost does not explicitly depend on the control input, from \cite[Theorem~3.3]{BK19} it follows that, for every $\mu \in \mathcal{P}(\Theta)$, any optimal control $u^\star\in \arg\min_\U \J$ satisfies the following condition:
    \begin{equation} \label{eq:PMP_ens}
        D^\star(t) \in \arg\max_{v\in [0,D_{\max}]} \left\{ \int_{\Theta} g^\theta_{D^\star}(t) \cdot F_1^\theta\big( X_{D^\star}^\theta(t) \big) v \de \mu(\theta) \right\}
    \end{equation}
    for a.e.~$t\in [0,T]$, where $t\mapsto g^\theta_{D^\star}(t)$ solves the \emph{backward adjoint equation}     
    \begin{equation*} 
        \begin{cases}
            \dot g_{D^\star}^\theta (t) = 
            -\ell' \left(N^\theta_{D^\star}(t) \right) (1,1) -g_{D^\star}^\theta(t) \left( \frac{\partial F_0^\theta (X_{D^\star}^\theta(t))}{\partial x} +
    \frac{\partial F_1^\theta (X_{D^\star}^\theta(t))}{\partial x} D^\star(t) \right),\\
        g_{D^\star}^\theta (T) = (0,\ldots,0),
        \end{cases}
    \end{equation*}
    for every $\theta\in\Theta$. We notice that, in the present setting, we only have normal extremals, as we do not impose any terminal-state constraint.
    From \cref{eq:PMP_ens}, we deduce the \emph{bang-bang} relation
    \begin{equation*}
        D^\star(t) = \begin{cases}
            D_{\max} & \mbox{if } \psi(t)>0, \\
            0 & \mbox{if } \psi(t)<0,
        \end{cases}
    \end{equation*}
    where $\psi(t) \coloneqq \int_{\Theta} g^\theta_{D^\star}(t) \cdot F_1^\theta\big( X_{D^\star}^\theta(t) \big) \de \mu(\theta)$. If the set $\{ t \in [0,T]: \psi(t) =0 \}$ has positive Lebesgue measure, the control $D^\star$ is said to be a \emph{singular arc}. For more details on singular arcs in averaged ensemble optimal control, we refer to \cite{aronna2025singular}.\\
    Finally, it is worth mentioning that, in the case of the minimax optimal control problem related to the minimization of \cref{eq:def_Funct_minimax}, the necessary optimaltiy condition for $D^\star_{\mathrm{mm}} \in \arg\min_\U (\J_{\max})$ looks rather similar to \cref{eq:PMP_ens}. Namely, we have that there exists $\bar \mu_{\max} \in \mathcal{P}(\Theta)$ such that
    \begin{equation} \label{eq:PMP_minimax}
        D^\star_{\mathrm{mm}} (t) \in \arg\max_{v\in [0,D_{\max}]} \left\{ \int_{\Theta} g^\theta_{D^\star_{\mathrm{mm}}}(t) \cdot F_1^\theta\big( X_{D^\star_{\mathrm{mm}}}^\theta(t) \big) v \de \bar \mu_{\max}(\theta) \right\}
    \end{equation}
    for a.e.~$t\in [0,T]$. We insist on the fact that $\bar \mu_{\max}$ is not given a priori, and it is in general hard to identify. For more details, see \cite{Vinter2005} and \cite[Remark~5.10]{Scag25}.
    In view of the analogy between \cref{eq:PMP_minimax,eq:PMP_ens}, we expect $D^\star_{\mathrm{mm}}$ to exhibit as well \emph{bang-bang} behavior.
\end{remark}
}

\subsection{Gradient of the objective functional}
In this part, we carry out the computations for the gradient of the functional $\J$ that we aim at minimizing. We report that the findings that we show below can be applied as well when we deal with an approximated functional $\J_{k}$ related to a measure $\mu_{k} \approx \mu$. 
{A main difficulty arises from the fact that admissible controls are constrained to the interval $[0,D_{\max}]$, so that the control space is convex but not linear; to address this, we develop here a careful argument.

For this reason,} it is convenient to introduce the functional $\J'\colon L^2([0,T], \R)\to \R$ defined according to \cref{eq:def_ens_funct} \emph{for every ${D} \in L^2([0,T], \R)$}. 
More precisely, we set
\begin{equation} \label{eq:def_ens_funct'}
    \J'({D})\coloneqq \int_\Theta  \int_0^T \ell\left({N_D^\theta(t)}\right)  
    \de {t} \de\mu(\theta),
\end{equation}
where {$N_D^\theta(t)=\sum_{i=1}^nX_{D\ i}^\theta(t)$, with $X_D^\theta = \left( X_{D\ 1}^\theta, \ldots, X_{D\ n}^\theta \right)$} solving
\begin{equation} \label{eq:cau_prob_trunc}
    \frac{\mathrm{d}}{\mathrm{d}{t}} {X_D^\theta}({t})
    =  F_0^\theta\left({X_D^\theta} ({t}) \right)
        +
         F_1^\theta\left({X_D^\theta} ({t}) \right) {D}({t}), \qquad     
         {X_{D\ i}}^\theta(0)  =
    {f_i N_0}
\end{equation}
for every $\theta\in\Theta$.
In other words, $\J'$ is the extension of $\J$ (which is defined only on $\U$) to the whole $L^2$, as illustrated in the next lemma. 
We observe that, in order to ensure that $\J'$ is well-defined for every ${D}\in L^2([0,T], \R)$, it is crucial to consider the truncated vector fields $F_0^\theta, F_1^\theta$ introduced in \cref{eq:def_trunc_fields} in place of $\tilde F_0^\theta, \tilde F_1^\theta$. 

\begin{lemma} \label{lem:extension}
    Let $\J\colon \U\to\R$ and $\J'\colon L^2([0,T],\R)\to \R$ be defined as in \cref{eq:def_ens_funct} and \cref{eq:def_ens_funct'}, respectively.
    If ${X_D^\theta}(0) = \left( {f_1N_0,\ldots, f_nN_0} \right)\in \Delta^{\theta}$ for every $\theta \in \Theta$, then $\J({D})=\J'({D})$ for every ${D} \in\U$.
\end{lemma}
\begin{proof}
    If ${D}\in\U$ and ${X_D^\theta}(0) = \left( {f_1N_0,\ldots, f_nN_0} \right)\in \Delta^{\theta}$ for every $\theta \in \Theta$, then \cref{lem:domain_invariance} implies that ${X_D^\theta(t)} \in \Delta^{\theta} {\subset \Delta}$ for every ${t}\in[0,T]$.
    Recalling that, for every $\theta\in\Theta$, $F_0^\theta \equiv \tilde F_0^\theta$ and $F_1^\theta \equiv \tilde F_1^\theta$ on $B_{2{K_{\max}}}(0)\supset \Delta$ (see \cref{eq:def_trunc_fields}), it turns out that $\J({D})=\J'({D})$.
\end{proof}

In view of \cref{lem:extension}, we now address the computation of the differential of $\J'$. 
The fact that $\J'$ is defined on the whole $L^2$ simplifies the arguments.
Given ${D} \in L^2([0,T],\R)$, we compute the differential of the functional $\J'$, and we represent it as an element of $L^2([0,T],\R)$ via the Riesz's isometry. 
Taking advantage of the results obtained in \cite{scag23gf}, we first consider the mapping $L^2([0,T],\R) \ni {D}\mapsto {X_D^\theta}(t) \in \R^n$ when ${t} \in [0,T]$ and $\theta\in\Theta$ are fixed. 

\begin{lemma} \label{lem:diff_traj}
    Let us consider ${D,\delta} \in L^2([0,T],\R)$ and $\e\in (0,1]$. Let us denote with ${X_D^\theta}, {X_{D+\e \delta}^\theta}$ the solutions of \cref{eq:cau_prob_trunc} corresponding, respectively, to the controls ${D, D+ \e\delta}$.
    Then, we have that
    \begin{equation}\label{eq:1st_ord_traj}
        \sup_{{t}\in [0,T]} \sup_{\theta\in \Theta}
        | {X_{D+\e \delta}}^\theta ({t}) - {X_D^\theta} ({t}) - \e {Y_{D, \delta}^\theta} ({t})
        | = o(\e) \quad \mbox{as } \e\to 0,
    \end{equation}
    where, for every ${t}\in[0,T]$ and for every $\theta\in\Theta$, we set
    \begin{equation} \label{eq:def_var_traj}
       {Y_{D, \delta}^\theta} ({t}) \coloneqq 
        M_{D}^\theta({t}) 
    \int_0^{t} 
    \left(M_{D}^{\theta}(\sigma) \right)^{-1}
    F_1^\theta \left( {X_D^\theta}(\sigma) \right) {\delta}(\sigma) 
    \de \sigma,
    \end{equation}
    and ${t}\mapsto M_{D}^\theta({t})\in \R^{{n}\times {n}}$ solves
    \begin{equation*}
        \begin{cases}
            \dot M_{D}^\theta(\sigma) =
            \left( \frac{\partial F_0^\theta ({X_D^\theta}(\sigma))}{\partial x} +
    \frac{\partial F_1^\theta ({X_D^\theta}(\sigma))}{\partial x} {D}(\sigma) \right) M_{D}^\theta(\sigma), \\
    M_{D}^\theta(0) = \mathrm{Id}.
        \end{cases}
    \end{equation*}
\end{lemma}
\begin{proof}
    See {Appendix}~\ref{sec:appendix_theory}.
\end{proof}

We are now ready for providing a representation of the differential of the functional $\J'$. {Here, we use $(\R^n)^*$ to denote the space of row-vectors, i.e., dual space of $\R^n$.}

\begin{proposition} \label{prop:grad_rep}
    Let ${\J'\colon L^2([0,T],\R)\to \R}$ be defined as in \cref{eq:def_ens_funct'}, and let us consider ${D,\delta}\in L^2([0,T],\R)$ and $\e\in [0,1]$.
    Then, the functional $\J'$ is Gateaux-differentiable at ${D}$, and we have that
    \begin{equation} \label{eq:grad_rep}
        \lim_{\e \to 0} \frac{\J'({D}+\e {\delta})- \J'({D})}{\e} =
        \int_0^T  
         \left( \int_\Theta
         {g_D^\theta}(\sigma) 
         F_1^\theta \left( {X_D^\theta}(\sigma) \right)
         \de \mu(\theta)
         \right)
         {\delta}(\sigma)
         \de \sigma,
    \end{equation}
    where ${g_D^\theta}\colon [0,T]\to (\R^{n})^*$ is an absolutely continuous curve that solves
    \begin{equation} \label{eq:grad_back_ode}
        \begin{cases}
             {\dot g_D^\theta} (\sigma) = 
            -\ell' \left({N_D^\theta}(\sigma) \right) (1,{\ldots,}1) -{g_D^\theta}(\sigma) \left( \frac{\partial F_0^\theta ({X_D^\theta}(\sigma))}{\partial x} +
    \frac{\partial F_1^\theta ({X_D^\theta}(\sigma))}{\partial x} {D}(\sigma) \right),\\
        {g_D^\theta} (T) = (0,{\ldots,}0),
        \end{cases}
    \end{equation}
    for every $\theta\in\Theta$.
\end{proposition}
\begin{proof}
    Recalling that by construction the vector fields $F_0^\theta, F_1^\theta$ vanishes outside the set $B_{3K_{\max}}(0)\subset \R^{n}$ for every $\theta\in\Theta$, we deduce that there exists $\kappa>0$ such that ${N_{D+\e \delta}^\theta} ({t}) = (1,{\ldots,}1)\cdot {X_{D+\e \delta}^\theta} \in [-\kappa,\kappa]$ for every $\theta\in\Theta$ and for every $\e\in [0,1]$.
    Recalling that $\ell\colon \R \to \R_+$ is of class $C^2$, owing to \cref{lem:diff_traj} we obtain that
    \begin{equation*}
        \sup_{{t}\in [0,T]}\sup_{\theta\in\Theta} \left|\ell \left({N_{D+\e \delta}^\theta}({t})\right) 
        - \ell \left({N_D^\theta}({t})\right) - \e \, \ell' \left({N_D^\theta}({t}) \right) (1,{\ldots,}1)\cdot {Y_{D,\delta}^\theta}({t}) \right| = o(\e) \quad \mbox{as } \e\to 0.
    \end{equation*}
    The last identity yields
    \begin{equation} \label{eq:Gateaux_prov_1}
        \J'({D}+\e {\delta}) = \J'({D}) + 
        \e \int_0^T \int_\Theta \ell' \left({N_D^\theta}({t}) \right) (1,{\ldots,}1)\cdot {Y_{D,\delta}^\theta}({t})
        \de \mu(\theta) \de {t}
        + o(\e) \quad 
    \end{equation}
    as $\e\to 0$.
    We now focus on the first order term in \cref{eq:Gateaux_prov_1}, and, taking advantage of \cref{eq:def_var_traj}, we compute:
    \begin{equation} \label{eq:comput_grad}
        \begin{split}
        &  \int_0^T  \int_\Theta  \ell' \left({N_D^\theta}({t}) \right) (1,{\ldots,}1) \cdot {Y_{D,\delta}^\theta}({t})
        \de \mu(\theta) \de {t}
         \\
         &  =
        \int_0^T \int_\Theta \ell' \left({N_D^\theta}({t}) \right) (1,{\ldots,}1)\cdot 
        M_{D}^\theta({t}) 
        \left[ \int_0^{t} 
    \left(M_{D}^{\theta}(\sigma) \right)^{-1}
    F_1^\theta \left( {X_D^\theta}(\sigma) \right) {\delta}(\sigma) 
    \de \sigma \right]
        \de \mu(\theta) \de {t} \\
        &  =
        \int_0^T \left[
        \int_\Theta 
         \int_\sigma^T 
        \ell' \left({N_D^\theta}({t}) \right) (1,{\ldots,}1)\cdot 
        M_{D}^\theta({t}) 
        \de {t} 
    \left(M_{D}^{\theta}(\sigma) \right)^{-1}
    F_1^\theta \left( {X_D^\theta}(\sigma) \right)  \de \mu(\theta) \right]
    {\delta}(\sigma) 
    \de \sigma,
        \end{split}
    \end{equation}
    where we used Fubini's Theorem in the second identity.
    For every $\theta\in\Theta$ and for every $\sigma\in [0,T]$, we define 
    \begin{equation} \label{eq:def_rep_grad_1}
        {g_D^\theta}(\sigma) \coloneqq \int_\sigma^T 
        \ell' \left({N_D^\theta}({t}) \right) (1,{\ldots,}1)\cdot 
        M_{D}^\theta({t}) 
        \de {t} 
    \left(M_{D}^{\theta}(\sigma) \right)^{-1},
    \end{equation}
    and, by combining \cref{eq:Gateaux_prov_1,eq:comput_grad,eq:def_rep_grad_1}, we deduce that \cref{eq:grad_rep} holds.
    We are left to show that ${g_D^\theta} \colon [0,T]\to (\R^{n})^*$ solves \cref{eq:grad_back_ode}. When $\sigma=T$, we directly read from \cref{eq:def_rep_grad_1} that the terminal condition ${g_D^\theta}(T)=(0,{\ldots,}0)$ is satisfied.
    By differentiating with respect to $\sigma$ the righ-hand side of \cref{eq:def_rep_grad_1}, we get 
    \begin{equation} \label{eq:der_rep_grad}
        \dot {g_D^\theta}(\sigma) =
        -\ell' \left({N_D^\theta}(\sigma) \right) (1,{\ldots,}1) + 
        \int_\sigma^T 
        \ell' \left({N_D^\theta}({t}) \right) (1,{\ldots,}1)\cdot 
        M_{D}^\theta({t}) 
        \de {t}
        \frac{\mathrm{d}}{\mathrm{d}\sigma}
        \left(M_{D}^{\theta}(\sigma) \right)^{-1}.
    \end{equation}
    Leveraging  a classical result (see, e.g., \cite[Theorem~2.2.3]{Bressan_Piccoli}), we obtain that
    \begin{equation*}
        \frac{\mathrm{d}}{\mathrm{d}\sigma}
        \left(M_{D}^{\theta}(\sigma) \right)^{-1} = -\left(M_{D}^{\theta}(\sigma) \right)^{-1}  \left( \frac{\partial F_0^\theta ({X_D^\theta}(\sigma))}{\partial x} +
    \frac{\partial F_1^\theta ({X_D^\theta}(\sigma))}{\partial x} {D}(\sigma) \right),
    \end{equation*}
    and finally, combining the last identity with \cref{eq:def_rep_grad_1,eq:der_rep_grad}, we finish the proof.
\end{proof}

\cref{prop:grad_rep} yields the following result for the differential of the functional $\J$ related to the model of interest.

\begin{corollary} \label{cor:rep_grad_J}
    Let $\J\colon \U\to \R$ be defined as in \cref{eq:def_ens_funct}, and let us consider ${D,D'}\in \U$ and $\e\in [0,1]$.
    Then, we have
    \begin{equation} \label{eq:grad_rep_J}
        \lim_{\e \to 0} \frac{\J({D+\e (D'-D)})- \J({D})}{\e} =
        \int_0^T  
         \left( \int_\Theta
        {g_D^\theta}(\sigma) 
         F_1^\theta \left( {X_D^\theta}(\sigma) \right)
         \de \mu(\theta)
         \right)
         {(D'(\sigma)-D(\sigma))}
         \de \sigma,
    \end{equation}
    where ${g_D^\theta} \colon [0,T]\to (\R^{n})^*$ solves \cref{eq:grad_back_ode} for every $\theta\in\Theta$.
\end{corollary}

\begin{proof}
    Recalling that $\U\subset L^2([0,T],\R)$ is convex, we have that ${D+\e(D'-D)}\in\U$ for every $\e\in[0,1]$. Hence, \cref{lem:extension} guarantees that $\J({D+\e(D'-D)})=\J'({D+\e(D'-D)})$ for every $\e\in[0,1]$. Therefore, we conclude using \cref{prop:grad_rep}.
\end{proof}

From \cref{cor:rep_grad_J} we readily get the representation of the differential of $\J$. Namely, for every ${D}\in\U$, owing to \cref{eq:grad_rep_J} we can represent through Riesz's isometry the differential $\nabla_{D}\J$ as
\begin{equation} \label{eq:grad_J}
    \nabla_{D} \J ({t}) \coloneqq 
    \int_\Theta
         {g_D^\theta}({t}) 
         F_1^\theta \left( {X_D^\theta}({t}) \right)
         \de \mu(\theta)
\end{equation}
for every ${t} \in [0,T]$.
We conclude this part by observing the structure of a gradient-based algorithm for the minimization of $\J$.
We first investigate the tangent cone to the set of admissible control $\U$ at a point ${D}\in\U$.
In doing this, we follow the notion provided in \cite[Definition~1.8]{mordukhovich2006variational1}, i.e.,
\begin{equation}\label{eq:def_tan_cone}
    T({D},\U) \coloneqq \limsup_{\e \to 0}\frac{\U- {D}}{\e},
\end{equation}
where the $\limsup$ in \cref{eq:def_tan_cone} is understood as the collection of the $L^2$-strong limiting points of sequences $(v_{\e_n})_n$ such that $v_{\e_n}\in \frac{\U- {D}}{\e_n}$ and $\e_n\to 0$ as $n\to\infty$.

\begin{lemma} \label{lem:tangent_domain}
    Let $\U \subset L^2([0,T],\R)$ be defined as in \cref{eq:def_adm_ctrls}. 
    Then, for every ${D}\in\U$, we have that 
    \begin{equation} \label{eq:char_tangent}
    T({D},\U) = \left\{ 
    v \in L^2([0,T],\R) :   v({t}) \leq 0 \,\,\, \mathrm{ if }\,\,\,  {D}({t})={D_{\max}},\,
    v(\tau) \geq 0 \,\,\, \mathrm{ if }\,\,\, {D}({t})=0
    \right\}    .
    \end{equation}
\end{lemma}
\begin{proof}
    See {Appendix}~\ref{sec:appendix_theory}.
\end{proof}

We introduce the mappings $\Pi_\U \colon L^2([0,T],\R) \to \U$ and $\Pi_{T({D},\U)} \colon L^2([0,T],\R) \to T({D},\U)$ defined as the projections onto the closed convex set $\U$ and onto the closed convex cone $T({D},\U)$, respectively.
With a direct computation, it is possible to show that
\begin{equation} \label{eq:proj_convex}
    \Pi_\U[v]({t}) = \max\big(\min\big( v({t}), {D_{\max}}\big), 0\big)
\end{equation}
and that
\begin{equation} \label{eq:proj_cone}
    \Pi_{T({D},\U)}[v]({t}) =   \max\big( v({t}), 0 \big) \mathds{1}_{\{{D}=0\}} + 
    \min\big( v({t}), 0 \big) \mathds{1}_{\{{D}={D_{\max}}\}}
    + v({t})\mathds{1}_{\{0<{D}<{D_{\max}}\}}
\end{equation}
for every ${t}\in[0,T]$, for every ${D}\in \U$ and for every $v\in L^2([0,T],\R)$.

\begin{proposition} \label{prop:g_descent}
    Let $\J\colon \U\to \R$ be defined as in \cref{eq:def_ens_funct}, and let us consider $\eta>0$. Then, for every ${D}\in \U$ we have that
    \begin{equation*}
        \Pi_\U\left[ {D} - \eta \nabla_{D} \J \right] =
        \Pi_\U\left[ {D} + \eta \Pi_{T({D},\U)} [-\nabla_{D} \J] \right].
    \end{equation*}
\end{proposition}
\begin{proof}
    To ease the notations, let us define $v_1 \coloneqq {D} - \eta \nabla_{D} \J$ and $v_2 \coloneqq {D} + \eta \Pi_{T({D},\U)} [-\nabla_{D} \J]$.
    From \cref{eq:proj_convex,eq:proj_cone}, we notice that $-\nabla_{D} \J$ can differ from $\Pi_{T({D},\U)} [-\nabla_{D} \J]$ only on those points ${t}\in[0,T]$ such that either ${D}({t})=0$ or ${D}({t})={D_{\max}}$. \\
    Let us assume that for some ${t}\in[0,T]$ we have ${D}({t})=0$. On the one hand, if $-\nabla_{D} \J({t})>0$, then $-\nabla_{D} \J({t}) = \Pi_{T({D},\U)} [-\nabla_{D} \J]$, yielding $v_1({t})=v_2({t})$ and $\Pi_\U[v_1]({t})= \Pi_\U[v_2]({t})$.
    On the other hand, if $-\nabla_{D} \J({t})<0$, then $\Pi_{T({D},\U)} [-\nabla_{D} \J]=0$, so that $v_2({t})=0$, while $v_1({t})<0$. However, using \cref{eq:proj_convex}, we deduce that $\Pi_\U[v_1]({t})= 0 = \Pi_\U[v_2]({t})$. \\
    The argument for ${t}$ such that ${D}({t})={D_{\max}}$ is analogous.
\end{proof}

\begin{remark} \label{rmk:g_descent}
The previous result suggests an implementable approach for the numerical minimization of the functional $\J$.
Given a current guess ${D}_n\in \U$, we perform the update ${D}_{n+1}\coloneqq \Pi_{\U}[{D}_n - \eta\nabla_{{D}_n}\J]$. 
\cref{prop:g_descent} ensures that we do not need to take care of projecting $-\nabla_{{D}_n}\J$ onto the tangent cone to $\U$ at ${D}_n$. {The convergence properties of the proposed projected gradient scheme (in particular, convergence to stationary points) follow from classical results on projected gradient methods in Hilbert spaces; see \cite[Chapter~2]{Bertsekas}.}
\end{remark}


\section{Ensemble optimal control in action: numerical experiments}
\label{sec:experiments}

In this section, we first describe the framework where we set our numerical experiments\footnote{{The codes used for the simulations are available at the following link: \url{https://doi.org/10.5281/zenodo.19652678}.}}, i.e., the androgen deprivation therapy in prostate cancer. 
Then, we present the benchmark strategies and we discuss the results obtained through the resolution of ensemble optimal control problems.
Finally, taking advantage of the insights provided by this viewpoint, we introduce a variant of the adaptive therapy proposed in \cite{moff2009adaptive,moff2017zhang}.

{In the remainder of this paper, we work within the two-population model mentioned in the Introduction (see \cref{eq:intro_model} and \cref{tab:simplified_parameters}).}
{As usual, in order to get a non-dimensional system of equations, we perform in \cref{eq:intro_model} the time re\-parametrization and the function transformations as follows:
\begin{equation} \label{eq:non_dim_transform}
\begin{split} 
 &\tau \coloneqq r_S t, \quad s(\tau) \coloneqq \frac{S(r_S t)}{K}, \quad r(\tau) \coloneqq \frac{R(r_S t)}{K},\\
  &n(\tau) \coloneqq \frac{N(r_S t)}{K}, \quad u(\tau) \coloneqq \frac{D(r_S t)}{D_{\max}}.
\end{split}
\end{equation}
We observe that in \cref{eq:intro_model} $r_S$ denotes the reproduction rate of the sensitive cells, and it is employed for the change in the time-scale. Consequently, when using the time variable $\tau$, the sensitive cells have unitary reproduction rate.
Finally, we rescale as well the time-independent parameters appearing in \cref{eq:intro_model}:
\begin{equation}\label{eq:scaled_param}
    \hat d_D \coloneqq 2 d_D, \quad \hat d_T \coloneqq \frac{d_T}{r_S}, \quad \hat r_R \coloneqq \frac{r_R}{r_S}.
\end{equation}
We collect in $\theta \coloneqq (\hat d_D, \hat d_T, \hat r_R, \hat f_0)$ the non-negative constants that parametrize the ensemble of systems described below, and that affect both the dynamics and the Cauchy datum.
Therefore, we obtain the following non-dimensional {equations}:
\begin{equation} \label{eq:intro_model_nondim}
\begin{split}
     &\begin{cases}
        \dot s^\theta_u(\tau) = \big( 1 - s^\theta_u(\tau) - r^\theta_u(\tau) \big) \left( 1 - \hat d_D u(\tau) \right)s^\theta_u(\tau) - \hat d_T s^\theta_u(\tau), \\
     \dot r^\theta_u(\tau) = \hat r_R \big( 1 - s^\theta_u(\tau) - r^\theta_u(\tau) \big) r^\theta_u(\tau) - \hat d_T r^\theta_u(\tau), 
    \end{cases} \\
     &\quad n^\theta_u(\tau) \coloneqq s^\theta_u(\tau) + r^\theta_u(\tau),
\end{split}
\end{equation}
where $\tau\mapsto u(\tau) \in [0,1]$ is the control which we act on the system with, and $\tau \mapsto s^\theta_u(\tau)$, $\tau \mapsto r^\theta_u(\tau)$ are the corresponding trajectories (see \cref{tab:nondim_parameters}). We denote with $n_0\in (0,1)$ the initial tumor size (i.e., $n^\theta(0)=n_0$), and we set the Cauchy datum for \eqref{eq:intro_model_nondim} to be
\begin{equation} \label{eq:intro_Cauchy}
    r^\theta_u(0) = r_0^\theta = \hat f_0n_0, \qquad s^\theta_u(0) = s_0^\theta =(1-\hat f_0)n_0, 
\end{equation}
with $\hat f_0 \in [0,1]$.
\begin{table}[ht]
\centering
\begin{tabular}{ll}
\toprule
\textbf{Parameter} & \textbf{Meaning} \\
\midrule
$s(\tau)$ & Normalized size of the sensitive subpopulation \\
$r(\tau)$ & Normalized size of the resistant subpopulation \\
$n(\tau)$ & Total normalized population $s(\tau)+r(\tau)$ \\
$\tau$ & Rescaled (dimensionless) time \\
$u(\tau)$ & Normalized drug dosage (control) \\
$\hat d_D$ & Rescaled drug sensitivity coefficient \\
$\hat d_T$ & Rescaled turnover (death) rate \\
$\hat r_R$ & Relative proliferation rate of resistant cells (with respect to sensitive cells) \\
$\hat f_0$ & Initial fraction of resistant cells \\
$n_0$ & Initial total normalized population \\
\bottomrule
\end{tabular}
\caption{Non-dimensional variables and parameters of the simplified model.}
\label{tab:nondim_parameters}
\end{table}
We insist on the fact that \cref{eq:intro_model_nondim} does not include a mechanism of secondary (or `acquired') resistance, and the population $r^\theta(\tau)$ at $\tau>0$ is made of descendants of the resistant clones $r_0^\theta$ present from the very beginning. This hypothesis is frequently assumed (see, e.g., \cite{cunningham2018optimal,turnover2021}) and is consistent with the bio-medical evidences concerning the heterogeneity of cancer cells \cite{swanton2012heterogeneity}. }

\subsection{Framework and parameters setting} \label{subsec:param_setting}

In this part, we consider the parameters appearing in \cref{eq:intro_model_nondim,eq:intro_Cauchy} and define the ranges where they vary.
In doing that, we take advantage of the estimates provided in \cite{turnover2021} for \emph{prostate cancer treated with androgen deprivation therapy (ADT)}.
We recall that the analysis in \cite{turnover2021} relied on the results of the trial described in \cite{bruchovsky2006final}. There, a cohort of patients with biochemical recurrence of the tumor (non-metastatic and castration sensitive, m0CSPC) after radical radiotherapy treatment underwent intermittent ADT.

\subsubsection*{Parameters not affected by uncertainty} 
We begin by listing the parameters that in our experiments {are assumed to be unaffected by uncertainty}. {This restriction is motivated by the fact that some parameters are taken from well-established estimates available in the literature. Furthermore, limiting the number of uncertain parameters allows us to study the variability in treatment response, while keeping the ensemble optimal control problem computationally tractable.}

The first crucial quantity to be estimated is the proliferation rate for sensitive cells $r_S$, which is employed to perform the time-normalization of the original system \eqref{eq:intro_model} into \eqref{eq:intro_model_nondim}.
In \cite{turnover2021}, the authors set $r_S=0.027\, \mathrm{day}^{-1}$, adopted from a previous estimate derived in \cite{moff2017zhang}.
Then, the effectiveness of the therapy is encoded in the value of the non-dimensional constant $\hat d_D$, which is set $\hat d_D=1.5$ in \cite{turnover2021} (see also \cite{moff2019multidrug} for the original estimate). 
We recall that the normalized dimension of the tumor is expressed using $n=n(\tau)$, which ranges in $[0,1]$. Here, $n=0$ means that the tumor cells (sensitive and resistant) are entirely extinct, while $n=1$ implies that the tumor size has reached the system's carrying capacity.
The initial size of the tumor $n(0)=n_0$ in our experiments takes values in $\{0.25, 0.50, 0.75\}\ni n_0$. We insist that $n_0$ is not affected by uncertainty, and we shall run the simulations separately for the different values of $n_0$.
Finally, in \cite{turnover2021} the turnover rate $\hat d_T$ is proposed to range in the interval $\hat d_T\in [0,0.5]$. In the present paper, we consider the most adverse scenario, i.e., we always set $\hat d_T=0$. 

\subsubsection*{Parameters affected by uncertainty}
We now describe the parameters that are affected by uncertainty.
The first quantity is the normalized proliferation rate of the resistant tumor cells encoded in $\hat r_R$. In \cite{turnover2021} this constant is suggested to vary in the interval $[0.5,1]\ni \hat r_R$. Here, $\hat r_R= 0.5$ means that the `cost' of having the drug resistance results in the proliferation rate being $50\%$ smaller than that of sensitive cells. Conversely, if $\hat r_R= 1$, the resistant cells have no disadvantage towards the sensitive ones regarding reproduction rate. \\
In the ODE model considered here, there is no mechanism of resistance acquirement during the {dynamics}. In other words, in our framework, the population of resistant cells $r(\tau)$ at an instant $\tau>0$ consists of clones of an initial resistant subpopulation that is assumed to be present from the very beginning.
This hypothesis is consistent with the heterogeneity and polyclonal nature of solid cancers, which in the clinical practice is particularly realistic in the case of advanced or metastatic disease (see e.g.~\cite{swanton2012heterogeneity}).
We assume that the initial fraction $\hat f_0$ of resistant cells ranges in $[0.002,0.1]\ni \hat f_0$, meaning that the ratio $\hat f_0\coloneqq r(0)/(r(0)+s(0))$ varies between $0.2\%$ (most favourable scenario) and $10\%$ (worst-case scenario).
In future work, we plan to incorporate a feature that accounts for the secondary resistance mechanism into the model.

\subsubsection*{Space of parameters and ensemble measure} We are in a position to provide a formal definition for the space of parameters $\Theta$ where $\theta\coloneqq (\hat d_D, \hat d_T, \hat r_R, \hat f_0)$ takes values.
According to what was said above, in our experiments we consider
\begin{equation*}
    \Theta = \{1.5\} \times \{0\} \times [0.5,1] \times [0.002,0.1].
\end{equation*}
In the current preliminary work, we assume the uncertain parameters $\hat r_R, \hat f_0$ to be \emph{independent} and \emph{uniformly distributed} in the respective intervals where they range.
This hypothesis implies that the probability measure $\mu\in \mathcal{P}(\Theta)$ used for defining the ensemble optimal control problem has the form:
\begin{equation} \label{eq:mu_simul}
    \mu \coloneqq \delta_{1.5} \otimes \delta_0 \otimes U(0.5,1) \otimes U(0.002,0.1),
\end{equation}
where $\delta_{z}$ denotes the Dirac delta centered at $z\in\R$, and $U(a,b)$ denotes the uniform probability distribution over the interval $[a,b]$.

\begin{remark} \label{rmk:indep_meas}
    The theoretical machinery developed in \cref{sec:theory} is completely general and does not require the independence of the different parameters that appear in the model.
    If, in the future, evidence of correlations between such parameters is observed, it will suffice to adapt accordingly the choice of the measure $\mu$.
    Similarly, if we wished to consider random fluctuations in $\hat d_D$ or $\hat d_T$, we could right away include these features in $\Theta$ and $\mu$.
\end{remark}

As observed in \cref{rmk:G_conv}, any ensemble optimal control problem related to the measure $\mu$ defined as in \cref{eq:mu_simul} requires the \emph{simultaneous} solution of infinitely many ODEs, just for evaluating the objective functional. 
Since this is unfeasible for simulations, we pursue the approach provided by \cref{thm:exist_Gamma}. Namely, we approximate $\mu$ with a discrete probability measure $\mu_k$ with finite support, and we address the minimization of the functional related to $\mu_k$.
{We emphasize that the construction of $\mu_k$ detailed below should be understood as a proof of concept. In a more realistic scenario, one can consider repeated measurements of tumor burden in a cohort of $k$ patients and estimate the model parameters using techniques from system identification (see, e.g., \cite{Nelles2020}). The model parameters for each patient, $\theta_1,\ldots,\theta_k$, are then interpreted as independent samples from the measure $\mu$, and the resulting empirical probability measure $
\mu_k \coloneqq \frac1k \sum_{i=1}^k \delta_{\theta_i}$
is used in the definition of the functional $\J_k$. 
A detailed description of such parameter estimation procedures is beyond the scope of the present manuscript and is left for future development. Our focus here is on the mathematical formulation and analysis of the ensemble optimal control problem, rather than on the practical implementation of parameter estimation.}\\
In our specific setting, we discretized the interval $[0.5,1]$ where $\hat r_R$ varies into $25$ equi\-spaced nodes with distance $0.02$. Hence, we set $U(0.5,1)\approx \frac{1}{25}\sum_{j=1}^{25} \delta_{\hat r_R^j}$, with $\hat r_R^j\in \{0.5, 0.52,\ldots,1\}$.
Regarding the parameter $\hat f_0$, we divided the interval $[0.002,1]$ into $49$ nodes with a constant step equal to $0.002$. Then, we considered $U(0.002,0.1)\approx \frac{1}{49}\sum_{k=1}^{49} \delta_{\hat f_0^k}$, with $\hat f_0^k\in \{0.002, 0.004,\ldots,0.1\}$.
Finally, we defined
\begin{equation} \label{eq:Theta_N_simul}
    \Theta_N \coloneqq \{1.5\} \times \{0\} \times \{0.5,0.52,\ldots,1\} \times \{0.002,0.004, \ldots, 0.1\},
\end{equation}
and
\begin{equation} \label{eq:mu_N_simul}
    \mu_N\coloneqq \delta_{1.5}  \otimes \delta_0 \otimes\frac{1}{25}\sum_{j=1}^{25} \delta_{\hat r_R^j} \otimes \frac{1}{49}\sum_{k=1}^{49} \delta_{\hat f_0^k},
\end{equation}
which has a support consisting of $N=1225$ points. From a practical perspective, the introduction of $\mu_N$ results in performing the numerical experiments on $N=1225$ different tumors, corresponding to $\theta^{j,k}=(1.5,0,\hat r_R^j, \hat f_0^k)$ with $j=1,\ldots,25$ and $k=1,\ldots,49$. We used these simulated tumors to evaluate every therapy schedule that we considered.

\subsubsection*{Numerical integration scheme}
For every treatment schedule that we tested, we approximated the dynamics of the control systems \eqref{eq:intro_model_nondim} using the Explicit Euler method with stepsize $\tau_{\mathrm{discr}}=r_S/8=3.375\cdot 10^{-3}$.
In the non-rescaled system \eqref{eq:intro_model}, this choice corresponds to dividing each simulation day into $8$ equi\-spaced time nodes.

\subsubsection*{Time-To-Progression (TTP)} 
In this paper, we measure the performances of the different strategies of drug scheduling using the Time-To-Progression (TTP), i.e., the length of the period from $\tau=0$ to the first instant $\tau_{\mathrm{TTP}}>0$ when the tumor is $20\%$ bigger (in terms of population) than the initial size $n_0$.
The disease is said `to be in progression' when this condition occurs.
More precisely, for every $\theta\in\Theta_N$ we set
\begin{equation} \label{eq:def_TTP}
    \tau_{\mathrm{TTP}}^\theta \coloneqq
    \inf_{\tau>0}\{ s^\theta(\tau) + r^\theta(\tau) \geq 1.2\, n_0 \}.
\end{equation}
We report that this quantity was employed also in \cite{turnover2021}, and we recall that is related to the radiologically-defined criterion `RECIST' (see \cite{eisenhauer2009RECIST}).
In view of the strategies resulting from ensemble optimal control problems, it is convenient to introduce a variant of TTP, which we denote with TTP$'$. We define it as follows:
\begin{equation} \label{eq:def_TTP'}
    \tau_{\mathrm{TTP'}}^\theta \coloneqq
    \sup_{\tau>0}\{ s^\theta(\tau) + r^\theta(\tau) \leq 1.2\, n_0 \}
\end{equation}
for every $\theta\in\Theta_N$.
\begin{remark} \label{rmk:TTPvsTTP'}
    Since $s^\theta(0) + r^\theta(0)=n_0< 1.2 \, n_0$, the continuity of the mapping $\tau\mapsto s^\theta(\tau) + r^\theta(\tau)$ for every $\theta$ implies that $\tau_{\mathrm{TTP}}^\theta \leq \tau_{\mathrm{TTP'}}^\theta $.
    If for some $\theta$ there exists a unique instant $\bar \tau >0$ such that $s^\theta(\bar \tau) + r^\theta(\bar \tau) = 1.2\, n_0$, then it turns out that $\tau_{\mathrm{TTP}}^\theta  = \tau_{\mathrm{TTP'}}^\theta = \bar \tau$.
\end{remark}

\subsection{Benchmark approaches: MTD and Adaptive Therapy} \label{subsec:benchmark}
We begin by describing the benchmark therapies for evaluating the performances of the schedules proposed later.
The baseline is the Maximal Tolerated Dose protocol (MTD), which consists of constantly giving the patient the maximal quantity of the drug. In our model, this results in setting $u_{\mathrm{MTD}}(\tau) = 1$ for every $\tau>0$.
When the competition between sensitive and resistant clones for shared resources is relevant (like in the present model), MTD has already been shown to be sub-optimal. For instance, the role played by this Darwinian mechanism was highlighted in \cite{moff2012evolutionary}.
Namely, the MTD leads rapidly to the extinction of the sensitive population, resulting in a quick and dramatic shrinking of the tumor size. 
After observing partial or total regression of the disease for a certain amount of time---whose duration depends indeed on the proliferation rate $\hat r_R$ of the resistant cells---, a resurgence of the tumor is typically observed, and the resistant cells are predominant. \\
Adaptive Therapy (AT) has been introduced in \cite{moff2009adaptive} for enhancing the long-time management of the disease, and it has been proven effective in several studies involving both numerical simulations \cite{moff2017zhang} and a pilot study with patients \cite{moff2017zhang,moff2022evolution}. 
The key idea is to adaptively introduce \emph{vacation periods} when drug treatment is discontinued.
Here, we apply the strategy as described in  \cite{turnover2021} on the ensemble of tumors modeled by $\Theta_N$ (see \cref{eq:Theta_N_simul}) and which we simulate the {dynamics} of. 
More precisely, the AT relies on the following steps, for every $\theta\in\Theta_N$:
\begin{enumerate}
\item At $\tau =0$, set {$u=1$} and keep it constant while $s^\theta(\tau) + r^\theta(\tau)> n_0/2$ (treatment period).
\item After observing $s^\theta(\tau)+r^\theta(\tau)\leq n_0/2$, reset {$u=0$} and keep it constant while  $s^\theta(\tau) + r^\theta(\tau)< n_0$ (vacation period).
\item After observing again $s^\theta(\tau) + r^\theta(\tau) \geq n_0$, reset {$u=1$} and go to Step~(1).
\end{enumerate}
To distinguish between this approach and the variant of AT that we will discuss later, we refer to the strategy just described as `On-Off'~AT.
The results are reported in \cref{table:MTD_AT1}.
We remark that in the experiments involving MTD and `On-Off' AT we observed $\tau_{\mathrm{TTP}}^\theta = \tau_{\mathrm{TTP'}}^\theta $ for every $\theta\in \Theta_N$ (see \cref{rmk:TTPvsTTP'}, and \cref{eq:def_TTP,eq:def_TTP'} for the definitions).

\begin{table}[h]
\scriptsize
\caption{MTD and `On-Off' AT.}
\label{table:MTD_AT1}

\centering
\begin{tabular}{c c c c c}
\toprule
& $n_0$ & max TTP & min TTP & mean TTP \\
\midrule
MTD & $0.25$ & $521$ days & $109$ days & $210$ days \\
`On-Off' AT & $0.25$ & $575$ days & $109$ days & $213$ days \\
\midrule
MTD & $0.50$ & $593$ days & $155$ days & $270$ days \\
`On-Off' AT & $0.50$ & $764$ days & $155$ days & $285$ days \\
\midrule
MTD & $0.75$ & $748$ days & $283$ days & $410$ days \\
`On-Off' AT & $0.75$ & $1196$ days & $283$ days & $462$ days \\
\bottomrule
\end{tabular}

\vspace{5pt}
\footnotetext{Comparison in terms of Time-to-Progression (TTP) between MTD and `On-Off'~AT. 
The mean TTP is computed by taking the average over the elements of the set $\Theta_N$ (see \cref{eq:Theta_N_simul}).}
\end{table}
\noindent
The values of the TTP in \cref{table:MTD_AT1} are perfectly consistent with the findings shown in \cite[Figure~3.A]{turnover2021}. Namely, on the one hand, the `On-Off'~AT never performed worse than the classical MTD. 
On the other hand, the advantage of the introduction of `treatment vacation periods' in AT is apparent when $n_0=0.5$ and $n_0=0.75$, where the progression of the disease is delayed respectively of $2$ and $7$ weeks on average, if compared to the mean TTP achieved using MTD.
In \cref{fig:MTD_AT1}, we plotted the {dynamics} of the tumor populations (sensitive, resistant, and total) when treated with MTD and `On-Off'~AT, for a specific value of the parameter $\theta= (\hat d_D, \hat d_T, \hat r_R, \hat f_0)$. For such a value of $\theta$, `On-Off'~AT dramatically outperforms MTD.
On the one hand, in MTD the sensitive sub\-population is rapidly extinct, and it is replaced by the resistant cells, which do not have any biological competitor for resources.
On the other hand, in `On-Off'~AT the `treatment vacation periods' allow for restocking the sensitive subpopulation, preventing its extinction. This results in an increased fight between sensitive and resistant cells for shared resources, contributing to the delay of the progression of the disease. Moreover, we notice that in `On-Off'~AT we can directly read the treatment-vacation cycles from the graph of the populations' {dynamics}.

\begin{figure}
    \centering
    \includegraphics[width=0.49\linewidth]{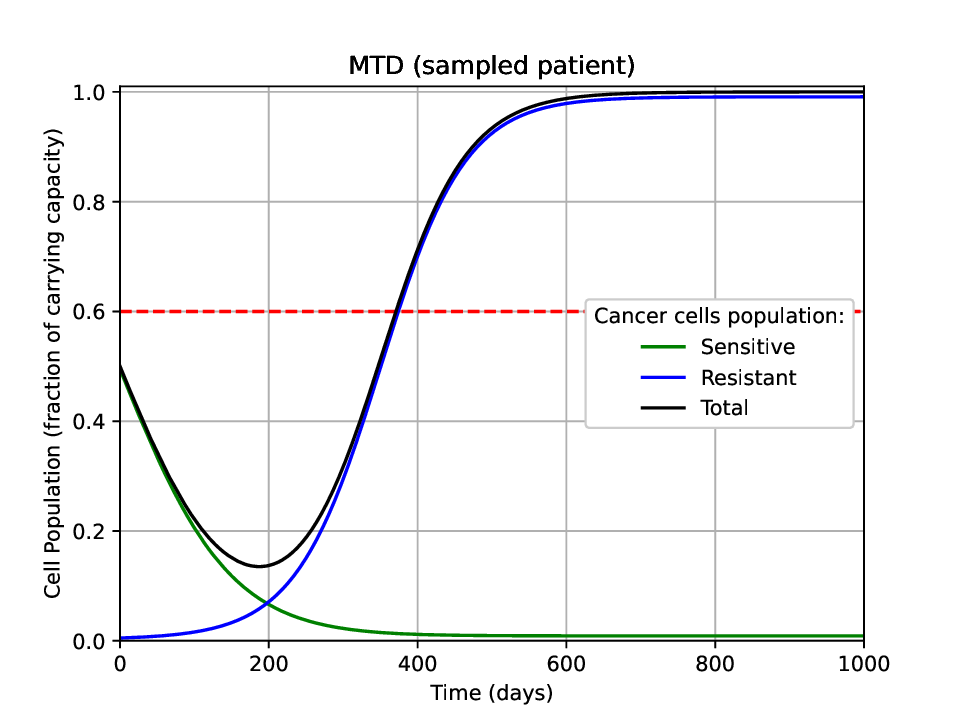}
    \includegraphics[width=0.49\linewidth]{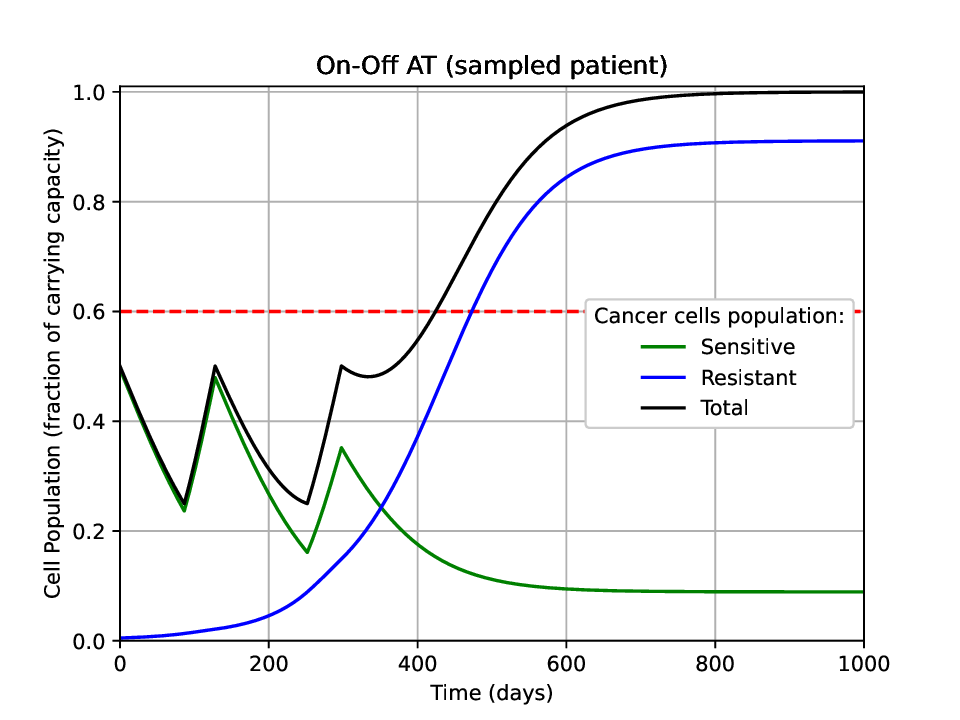}
    \caption{\footnotesize {Dynamics} of the tumor corresponding to $\theta = (1.5,0,0.66,0.01)$ and with initial size $n_0=0.5$ undergoing MTD (left) and `On-Off'~AT (right).
    The dashed horizontal line represents the threshold tumor size related to the condition `cancer in progression'.
    For this $\theta$, the advantage of `On-Off'~AT (right picture) over MTD (left picture) is apparent. The TTP of `On-Off'~AT and MTD is $424$ days and $370$ days, respectively, with a progression delay of almost $8$ weeks.}
    \label{fig:MTD_AT1}
\end{figure}

\subsection{Ensemble optimal control problems: formulation and resolution} \label{subsec:ens_OC_form}
In this part, we consider the problem of finding a therapy schedule $u=u(\tau)$ for the control system \eqref{eq:intro_model_nondim}, in the case some of the parameters appearing in \eqref{eq:intro_model_nondim} are affected by uncertainty. We insist that, when solving an ensemble optimal control problem, we aim to find a shared policy $u=u(\tau)$ that shall be used \emph{for every admissible tuple of parameters}.
In \cref{subsec:param_setting} we introduced the space of parameters and the discrete measure $\mu_N$. Before addressing the resolution of the ensemble optimal control problem, we are left to set the time horizon for the controlled {dynamics} and the function $\ell\colon \R\to\R$ that designs the integral cost.

\vspace{6pt}

\begin{paragraph}{{Simulation} interval}
    For every $n_0\in\{0.25,0.50,0.75\}$ we set the value of the time horizon $T_{\mathrm{hor}}$ according to the maximal value of TTP observed in \cref{table:MTD_AT1}. 
    We collect the selected values in \cref{table:Time_horizon}, where we also reported the normalized time $T\coloneqq T_{\mathrm{hor}} \cdot r_S$ that is required in the definition of the functional $\J$ (see \cref{eq:def_ens_funct}). 
    We recall that $r_S=0.027 \,\mathrm{days}^{-1}$ denotes the proliferation rate of the sensitive cells, and it has been used for the time-normalization of the original system \eqref{eq:intro_model} into \eqref{eq:intro_model_nondim}.  
    
\begin{table}[h]
\caption{Time horizon.}
\label{table:Time_horizon}

\centering
\begin{tabular}{c c c}
\toprule
$n_0$ & $T_{\mathrm{hor}}$ & $T$ \\
\midrule
$0.25$ & $750$ days & $20.25$ \\
$0.50$ & $1000$ days & $27.00$ \\
$0.75$ & $1500$ days & $40.50$ \\
\bottomrule
\end{tabular}

\vspace{5pt}
\footnotetext{Values of $T_{\mathrm{hor}}$ used for the optimal control problems.}
\end{table}

\begin{remark} \label{rmk:choice_T}
    The choice of $T_{\mathrm{hor}}$ in \cref{table:Time_horizon} may sound arbitrary, and to some extent, it is. For every $n_0$, we set $T_{\mathrm{hor}}$ to be $\approx 30\%$ larger than the best TTP obtained through `On-Off'~AT.
    However, as we will discuss in detail later in \cref{rmk:time_san_check}, the resolution of an ensemble optimal control problem tends automatically to find the interval where the decision-making on the therapy schedule has the most significant effect.
\end{remark}
\end{paragraph}

\vspace{6pt}

\begin{paragraph}{Integral cost design}
    A crucial step in resolving any (ensemble) optimal control problem is the definition of the objective functional to be minimized.
    In view of providing an explicit expression for the functional $\J\colon \U\to \R$, we need to specify the function $\ell\colon \R\to \R$ involved in the integral cost.
    Here, we propose and compare two possible alternatives for $\ell$. 
    The first natural attempt consists of setting
    \begin{equation} \label{eq:ell_cost_1}
        \ell^1(n) \coloneqq n-n_0,
    \end{equation}
    which results in linearly penalizing any deviation of the tumor size $n^\theta(\tau) = s^\theta(\tau) + r^\theta(\tau)$ from the initial condition $n_0$.
    We observe that $\ell^1$ has a symmetric behavior around $n_0$: If $\delta>0$ denotes a variation of the tumor size, a decrease amounting to $-\delta$ (i.e., $n=n_0-\delta$) is awarded a `negative cost' $-\delta$, which in absolute value is as much as $\ell^1$ penalizes the growth $n=n_0+\delta$.
    This property of $\ell^1$ does not fully reflect the physicians' aim when treating a patient for long-term disease control.
    Indeed, \emph{the goal is to stabilize the tumor as long as possible rather than to eradicate it}. 
    For this reason, we consider also the hyperbolic function $\ell^2$ defined as follows:
    \begin{equation} \label{eq:ell_cost_2}
        \ell^2(n) \coloneqq \sqrt{1 +(n-n_0)^2} -1 + (n-n_0),
    \end{equation}
    whose behavior around $n_0$ is not symmetric (see \cref{fig:costs_comp}).
    In this way, a positive deviation is penalized more than the corresponding negative deviation is awarded.\\
    Nonetheless, the design of a proper integral cost for the long-term cancer management is an interesting and delicate problem that goes beyond the scope of the present paper. We leave open this point for future developments.  

    \begin{figure}
        \centering
        \includegraphics[width=0.5\linewidth]{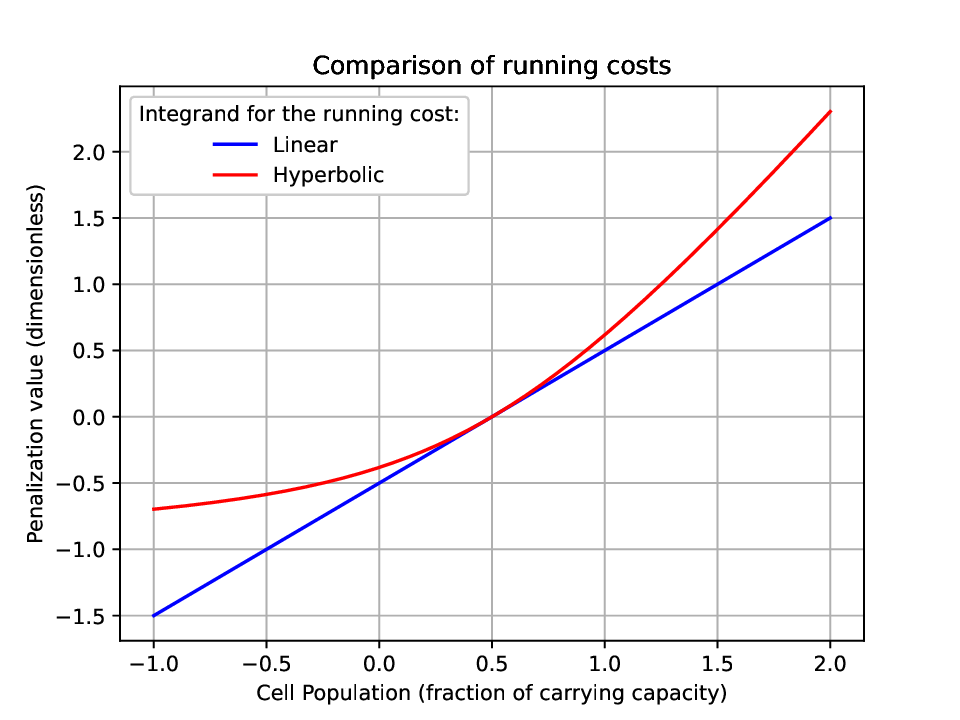}
        \caption{\footnotesize Graph of $\ell^1$ (linear) and $\ell^2$ (hyperbolic) for $n_0=0.50$ (see, respectively, \cref{eq:ell_cost_1,eq:ell_cost_2}).}
        \label{fig:costs_comp}
    \end{figure}
\end{paragraph}

\subsubsection*{Minimization of the cost functionals}
For every $n_0\in \{0.25,0.5,0.75\}$ we considered the functionals $\J^1,\J^2\colon \U \to \R$ defined as follows:
\begin{equation} \label{eq:funct_simul}
    \J^1(u) \coloneqq  \int_0^T 
    \int_\Theta  \ell^1 \left(n_u^\theta(\tau )\right)  \de\mu_N(\theta)
    \de \tau, \qquad
    \J^2(u) \coloneqq  \int_0^T 
    \int_\Theta  \ell^2 \left(n_u^\theta(\tau )\right)  \de\mu_N(\theta)
    \de \tau,
\end{equation}
where $\ell^1,\ell^2$ have been introduced in \cref{eq:ell_cost_1,eq:ell_cost_2}, $T$ is set according to \cref{table:Time_horizon}, and $\mu_N$ is the discrete measure on the ensemble of parameters that we defined in \cref{eq:mu_N_simul}.
We used the following gradient-based minimization scheme for the numerical minimization of $\J^1$ and $\J^2$ (see \cref{prop:g_descent}):
\begin{equation} \label{eq:g_descent_simul}
    u^i_{k+1} \gets \Pi_\U \left[ u_k^i - \eta \nabla_{u_k^i} \J^i \right]\qquad k\geq 0, \quad i=1,2,
\end{equation}
where $\Pi_\U\colon  L^2([0,T], \R) \to \U$ is the projection onto $\U$ (see \cref{eq:proj_convex}), and where we set $\eta=0.125$.
In order to avoid the introduction of any bias towards turning the therapy on or off, we considered as an initial guess the control $u^i_0\equiv 0.5$ for $i=1,2$, which is at every instant equi\-distant from $1$ (full-dosage) and $0$ (discontinued therapy).
We repeated the step described in \cref{eq:g_descent_simul} for $500$ iterations. 
We implemented the simulations in Python, and we relied on the automatic differentiation tools of Pytorch.
{ We also performed numerical simulations on the minimax problems related to the functionals functionals $\J^1_{\max},\J^2_{\max} \colon \U \to \R$ defined as follows:
\begin{equation*} 
    \J^1_{\max} (u) \coloneqq   \max_{\theta \in \Theta_N} \left\{ \int_0^T 
      \ell^1 \left(n_u^\theta(\tau )\right) 
    \de \tau \right\}, \quad
    \J^2_{\max} (u) \coloneqq  \max_{\theta \in \Theta_N} \left\{\int_0^T  \ell^2 \left(n_u^\theta(\tau )\right)  \de \tau \right\}.
\end{equation*}
For this task, we employed the subgradient-like method proposed in \cite[Algorithm~1]{Scag25}.
}

\begin{remark} \label{rmk:time_san_check}
    From \cref{fig:simul_ctrls} we notice that, for $\tau$ close to $T$, the computed controls $u^1_{500}$ and $u^2_{500}$ for $n_0=0.50$ do not deviate significantly from the value of the initial guess $u^1_{0}=u^2_0 \equiv 0.5$. We observed a similar behavior for $n_0=0.25$ and $n_0=0.75$ as well.
    This phenomenon is since, throughout the iterations of the gradient descent $k=0,\ldots,500$, $\nabla_{u_k^1} \J^1(\tau)$ and $\nabla_{u_k^2} \J^2(\tau)$ are small when $\tau$ is close to $T$.
    From the model perspective, this suggests that, at later stages of the {simulation} horizon (i.e., when $\tau$ is close to $T$), the differences in the treatment strategy do not have a relevant impact on the final outcome.
    Intuitively, if for $\tau\geq \bar \tau$ most of the tumors of the ensemble $\Theta_N$ have already progressed to the carrying capacity of the system, then we expect that any change in the treatment policy after the instant $\bar \tau$ will not drastically improve the result.
    We can use this phenomenon as an empirical test for deciding whether the time horizon $T$ has been appropriately set large (see \cref{rmk:choice_T}).
    {Finally, we notice that the same phenomenon can be observed in the graph of the computed controls for $\J^1_{\max}$ and $\J^2_{\max}$.}
\end{remark}

\begin{figure}
    \centering
    \includegraphics[width=0.49\linewidth]{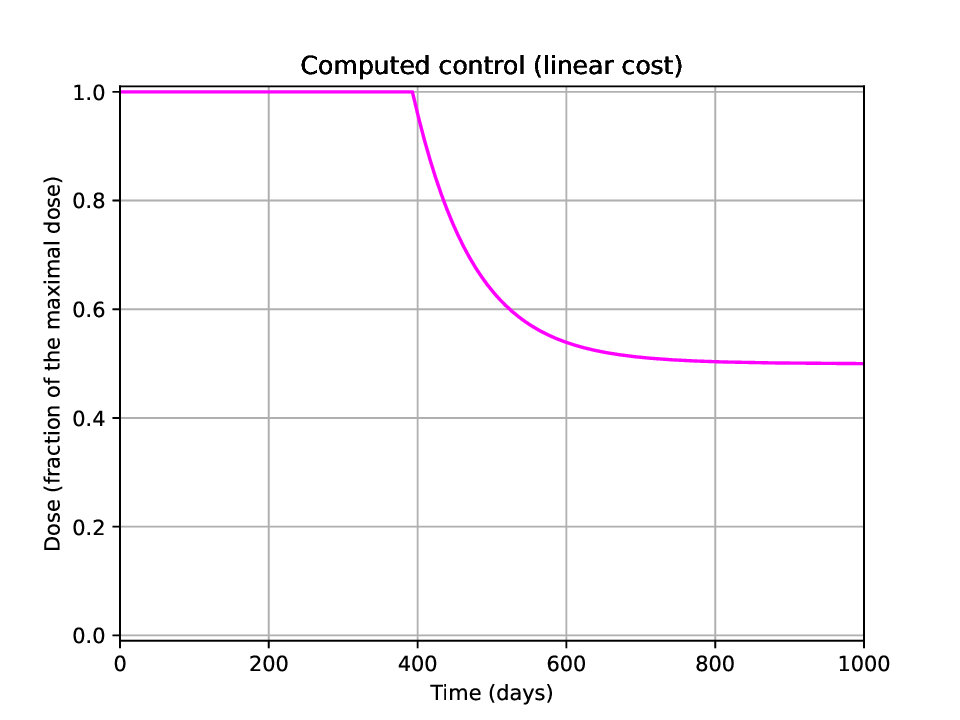}
    \includegraphics[width=0.49\linewidth]{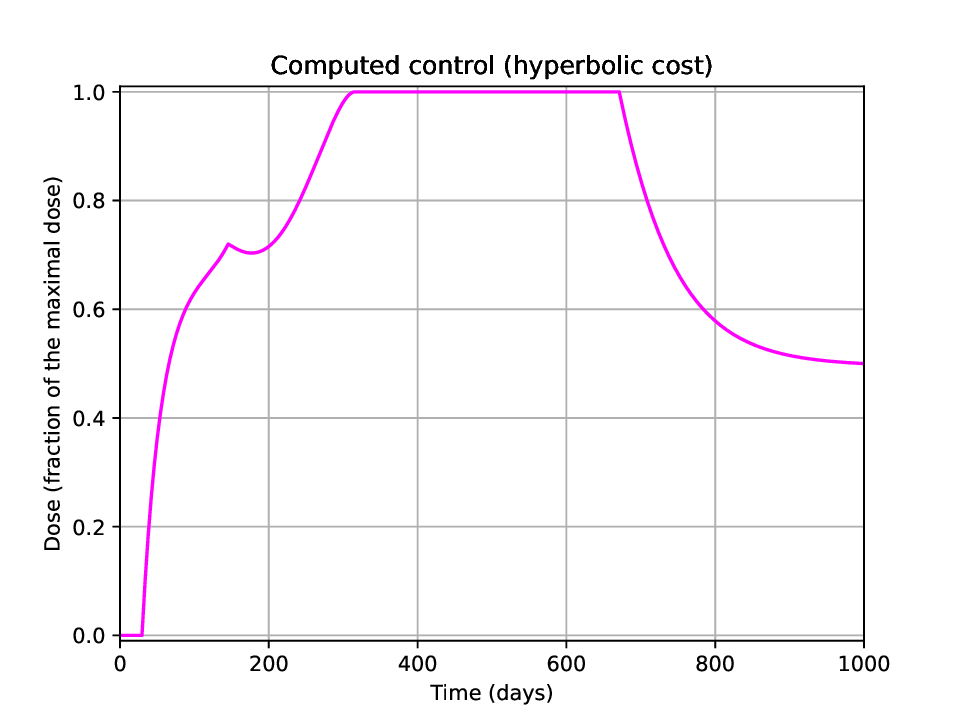}
    \caption{\footnotesize Controls computed by minimizing $\J^1$ (left) and $\J^2$ (right) for $n_0=0.50$. The profiles corresponding to $n_0=0.25$ and $n_0=0.75$ are qualitative similar.}
    \label{fig:simul_ctrls}
\end{figure}

\begin{figure}
    \centering
    \includegraphics[width=0.49\linewidth]{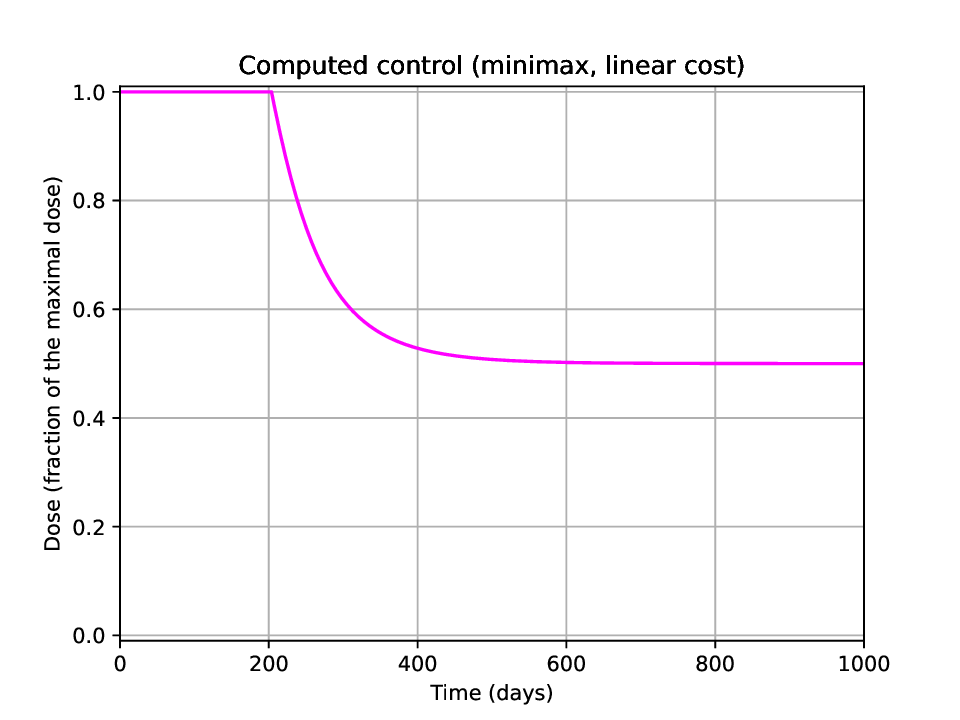}
    \includegraphics[width=0.49\linewidth]{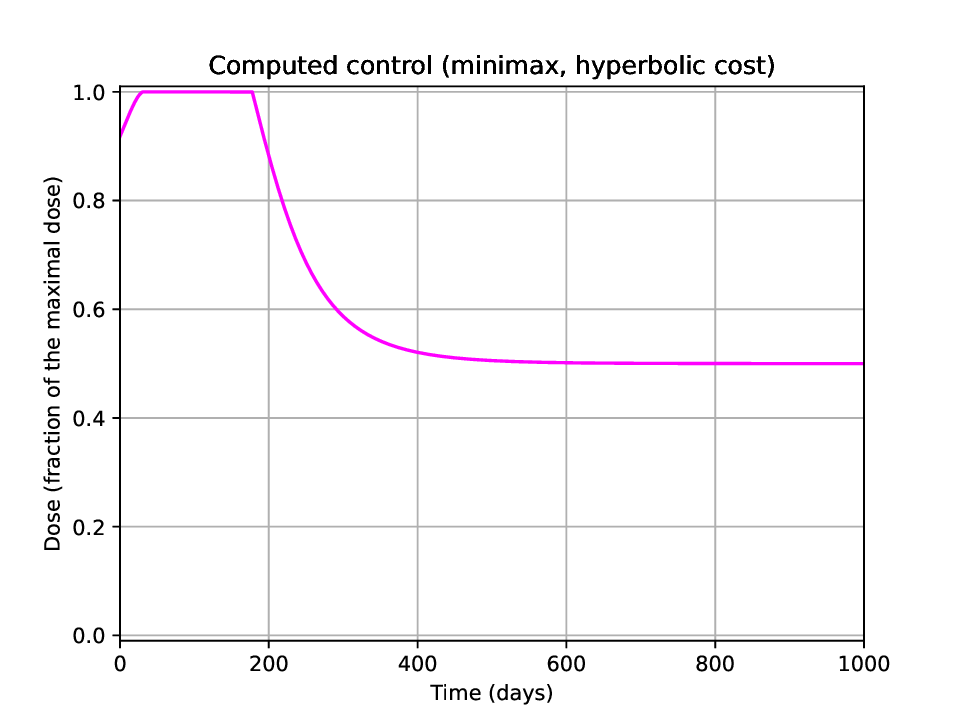}
    \caption{\footnotesize Controls computed by minimizing $\J^1_{\max}$ (left) and $\J^2_{\max}$ (right) for $n_0=0.50$. The profiles corresponding to $n_0=0.25$ and $n_0=0.75$ are qualitative similar.}
    \label{fig:simul_ctrls_minimax}
\end{figure}

\subsubsection*{Results}
We first show in \cref{fig:simul_ctrls} the profiles of the controls computed via the gradient-based algorithm outlined above. 
Interestingly, the integral costs $\ell^1$ and $\ell^2$ tend to select optimal controls showing rather different qualitative behavior.
Indeed, on the one hand, the computed optimal control for the functional $\J^1$---which incorporates the \emph{linear} size penalization of the tumor---suggests giving the patient the maximal drug dose for a rather long initial interval. For instance, as we can read from \cref{fig:simul_ctrls}, at the end of the optimization procedure for $n_0=0.50$, $u^1_{500}$ is constantly equal to $1$ in the first $\approx 400 \,\mathrm{days}$. 
For this reason, when using this strategy, we expect to observe performances in terms of $\mathrm{TTP}$ very close to the ones reported in \cref{table:MTD_AT1} for MTD.
On the other hand, the \emph{hyperbolic} size penalization involved in the definition of $\J^2$ induces an entirely different profile in the computed control. 
Indeed, $u^2_{500}$ shows an initial phase where the therapy is completely turned off, and the tumor can grow freely. 
The rationale behind this choice is to increase the size of the sensitive cells in the first part of the therapy to further promote the fight for resources between the two sub\-populations.
In \cref{fig:Sum_Hyp} we show the {progression} of the disease when adopting the computed controls $u^1_{500}$ and $u^2_{500}$ for a tumor with the same parameters as in \cref{fig:MTD_AT1}. 

\begin{figure}
    \centering
    \includegraphics[width=0.49\linewidth]{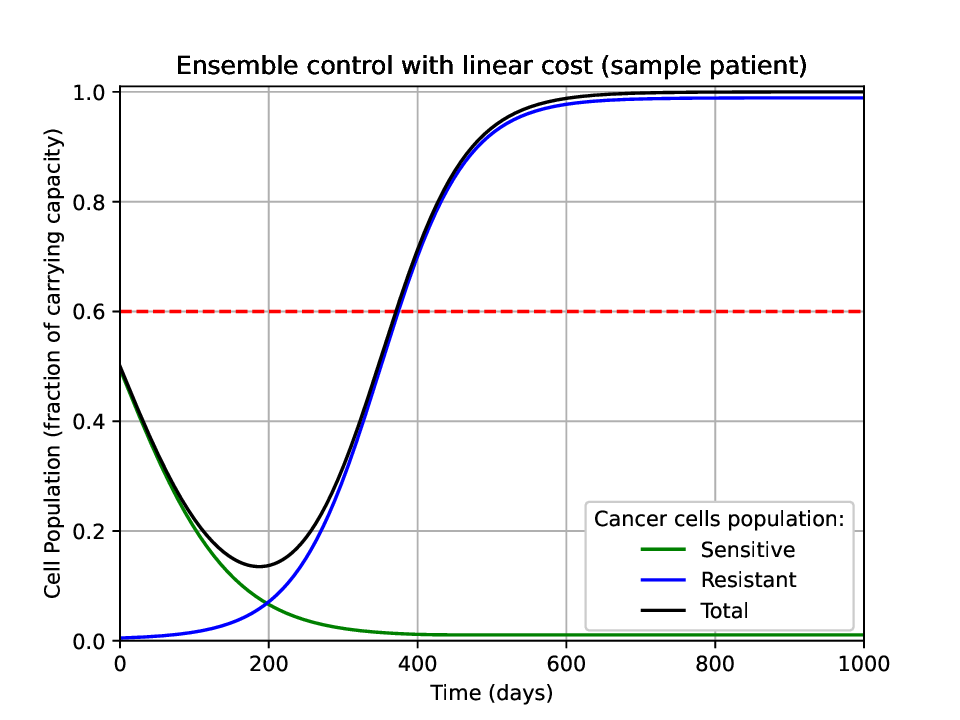}
    \includegraphics[width=0.49\linewidth]{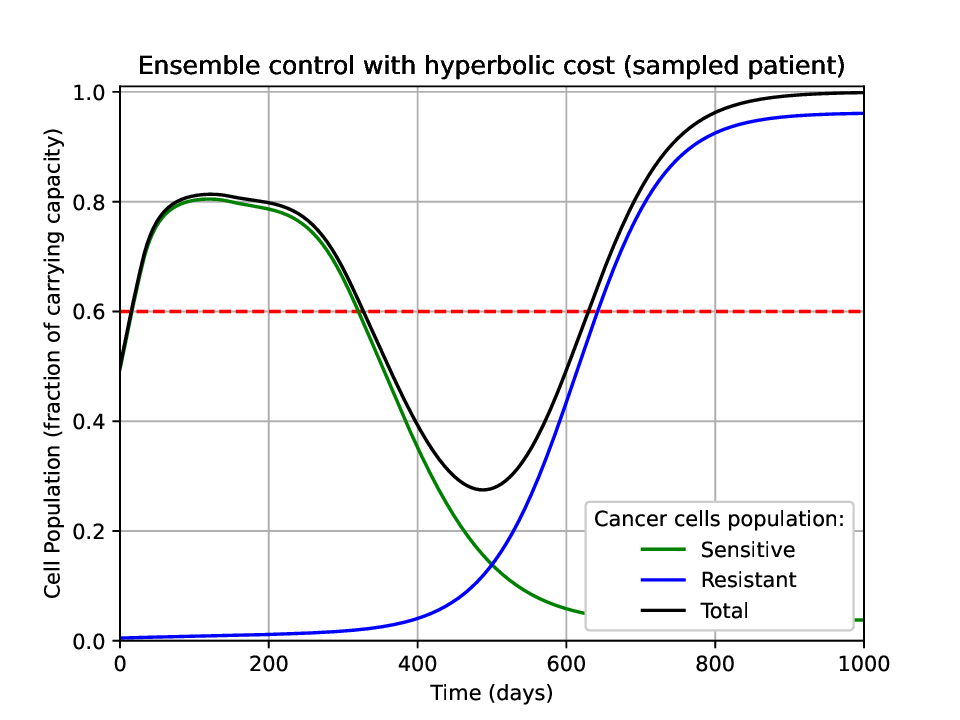}
    \caption{\footnotesize {Dynamics} of the tumor corresponding to $\theta = (1.5,0,0.66,0.01)$ and with initial size $n_0=0.5$ using the treatment prescribed by the approximated minimizers of $\J^1$ (linear cost, left) and of $\J^2$ (hyperbolic cost, right).
    The dashed horizontal line represents the threshold tumor size related to the condition `cancer in progression'.}
    \label{fig:Sum_Hyp}
\end{figure}

\noindent
The graphs in \cref{fig:Sum_Hyp} reflect the different qualitative behavior of the computed minimizers of $\J^1$ and $\J^2$.
Namely, on the one hand, we notice that the picture related to $\J^1$ (left) is almost indistinguishable from the {progression} under the MTD approach (see the picture on the left-hand side in \cref{fig:MTD_AT1}).
On the other hand, the control obtained by minimizing $\J^2$ (right) succeeds in delaying the growth of the resistant sub\-population. However, in the first part of the treatment, the total tumor size exceeds the progression threshold.
Hence, it turns out that for such a strategy $\tau_{TTP}^\theta < \tau_{TTP'}^\theta$ (see \cref{eq:def_TTP,eq:def_TTP'} for the definitions). 
Indeed, when the `progression size-level' is crossed for the first time, the tumor is mainly made of sensitive cells, which are killed at a proper later stage.
{Finally, we observe that the computed controls for the minimax formulation (i.e., related to $\J^1_{\max}, \J^2_{\max}$) show in both cases an initial phase where they are close to $1$, with only minor differences between them. For this reason, these scheduling are expected to behave similarly as the MTD regime.}\\
In \cref{table:Sum_Hyp} we report the TTPs resulting from the computed schedules. We insist on the fact that for having a fair evaluation of the TTP, in the rows marked as `Hyperbolic' (corresponding to the cost $\ell^2$), we used $\tau_{TTP'}^\theta$. 
Moreover, when testing the policy related to the linear cost $\ell^1$, we observed $\tau_{TTP}^\theta = \tau_{TTP'}^\theta$ for every $\theta\in \Theta_N$.

\begin{table}[h]
\caption{Strategies related to linear and hyperbolic cost.}
\label{table:Sum_Hyp}

\centering
\begin{tabular}{c c c c c}
\toprule
& $n_0$ & max TTP & min TTP & mean TTP \\
\midrule

Lin.~(averaged) & $0.25$ & $521$ days & $109$ days & $210$ days \\
Hyp.~(averaged) & $0.25$ & $672$ days & $9$ days & $102$ days \\
Lin.~(minimax)  & $0.25$ & $506$ days & $109$ days & $209$ days \\
Hyp.~(minimax)  & $0.25$ & $499$ days & $107$ days & $210$ days \\
\midrule

Lin.~(averaged) & $0.50$ & $593$ days & $155$ days & $270$ days \\
Hyp.~(averaged) & $0.50$ & $850$ days & $15$ days & $336$ days \\
Lin.~(minimax)  & $0.50$ & $584$ days & $155$ days & $267$ days \\
Hyp.~(minimax)  & $0.50$ & $580$ days & $155$ days & $267$ days \\
\midrule

Lin.~(averaged) & $0.75$ & $747$ days & $283$ days & $410$ days \\
Hyp.~(averaged) & $0.75$ & $968$ days & $507$ days & $660$ days \\
Lin.~(minimax)  & $0.75$ & $752$ days & $282$ days & $402$ days \\
Hyp.~(minimax)  & $0.75$ & $755$ days & $278$ days & $400$ days \\
\bottomrule
\end{tabular}

\vspace{5pt}
\footnotetext{Comparison in terms of Time-to-Progression (TTP) between the schedules obtained by minimizing $\J^1, \J^1_{\max}$ (linear) and $\J^2, \J^2_{\max}$ (hyperbolic). The mean TTP is computed by taking the average over the elements of the set $\Theta_N$ (see \cref{eq:Theta_N_simul}). We insist on the fact that in the rows marked as `Hyperbolic' we used $\tau_{TTP'}^\theta$.}
\end{table}

\noindent
The results in \cref{table:Sum_Hyp} confirm that the schedule obtained by minimizing $\J^1$ (linear cost) is substantially equivalent to MTD since the measured TTPs are almost identical. {The same observation holds as well for the controls related to the minimax functionals $\J^1_{\max},\J^2_{\max}$.}
As for the policy related to $\J^2$ (hyperbolic cost), interpreting the results is more complicated.
Indeed, on the one hand, for $n_0=0.50$ and $n_0=0.75$ the computed mean~TTPs ($336$ days and $660$ days, respectively) are promising and show an apparent improvement when compared to the benchmark strategy `On-Off'~AT (see \cref{table:MTD_AT1}).
On the other hand, for $n_0=0.25$, the policy related to $\J^2$ performs poorly since the mean TTP has more than halved if compared to the other treatment approaches.
In addition, for $n_0=0.25$ and $n_0=0.50$, we observe a dramatic drop in min~TTP, i.e., the time-to-progression for the worst-case tumor.
In \cref{fig:worst_hyp} we reported the evolution of the tumors with $n_0=0.25$ and $n_0=0.50$ for which $\tau_{\mathrm{TTP}'}^\theta$ attains the minimal value ($9$~days and $15$~days, respectively).

\begin{figure}
    \centering
    \includegraphics[width=0.49\linewidth]{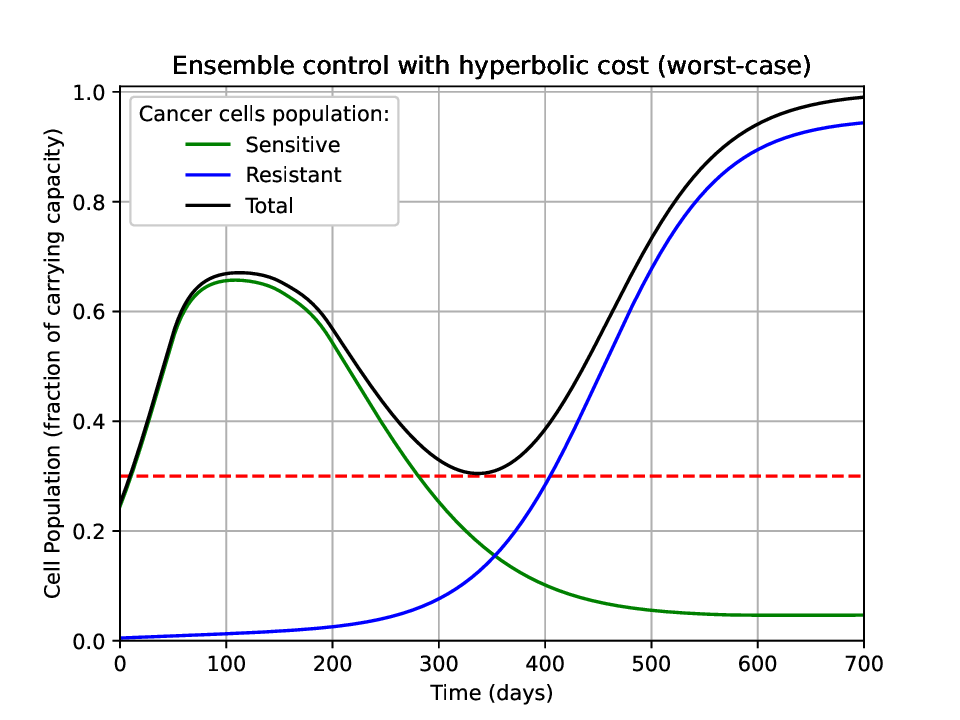}
    \includegraphics[width=0.49\linewidth]{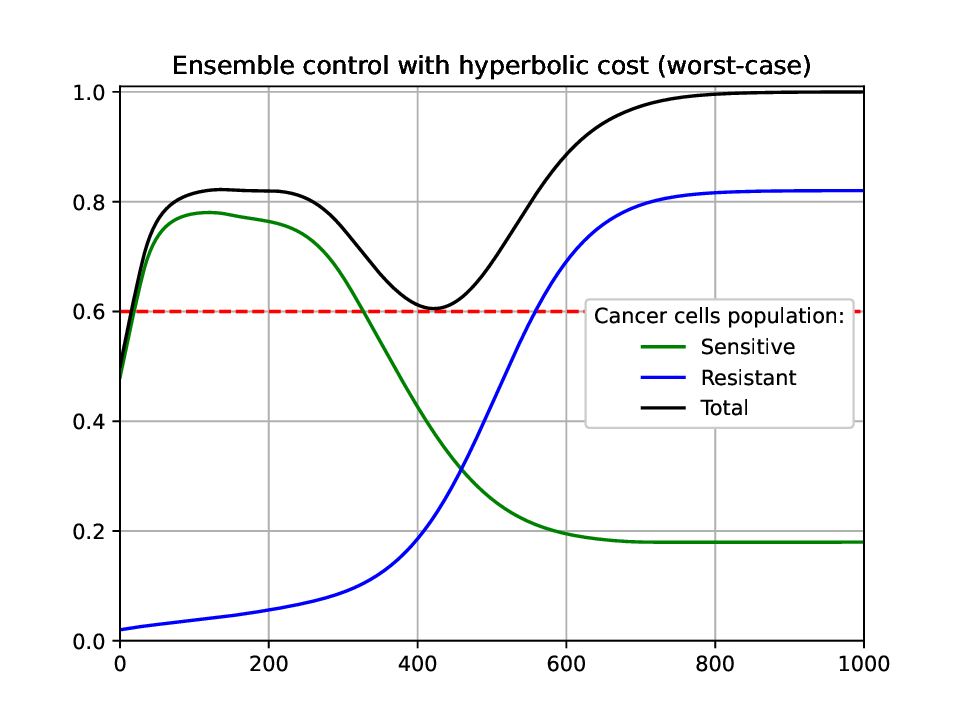}
    \caption{\footnotesize Worst-case tumors for the schedules computed by minimizing $\J^2$ for $n_0=0.25$ ($ \tau^\theta_{\mathrm{TTP}'} = 9$~days, left) and $n_0=0.50$ ($\tau^\theta_{\mathrm{TTP}'} = 14$~days, right). The tumors corresponding to these graphs have, respectively, parameters $\theta=(1.5,0,0.72,0.02)$ (left) and $\theta=(1.5,0,0.84,0.04)$ (right).}
    \label{fig:worst_hyp}
\end{figure}

\noindent
We notice that the {dynamics} of the total tumor population is qualitatively similar in the two graphs in \cref{fig:worst_hyp}.
Namely, the computed policy prescribes an initial phase where the tumor can freely grow. However, differently from what is observed in the right-hand side picture of \cref{fig:Sum_Hyp}, the total population size never shrinks below the progression threshold.

\subsubsection*{Ensemble optimal control: conclusions} 
When solving ensemble optimal control problems for designing drug schedules, the choice of the cost that penalizes the tumor size plays a crucial role.
On the one hand, the linear penalization led to outcomes substantially identical to MTD. On the other hand, the hyperbolic cost introduced in \cref{eq:ell_cost_2} showed promising results. 
In the latter case, the strategy is to increase the number of sensitive cells in the early stage of therapy to accentuate the fight for resources in the two sub\-populations.
{A limitation of the resulting optimal policy is that it allows the tumor to grow substantially beyond the progression threshold before initiating any reduction in tumor size (see the right-hand side of \cref{fig:Sum_Hyp}). Moreover, in some scenarios, even during the active treatment phase, the tumor size does not decrease below the progression threshold (see \cref{fig:worst_hyp}).}
{Motivated by} these observations, we propose a variant of Adaptive Therapy in the next subsection.

\subsection{`Off-On' Adaptive Therapy} \label{subsec:Off-On}
In this part, we suggest a variant for the `On-Off'~AT detailed in \cref{subsec:benchmark}.
The source of inspiration is the policies related to the hyperbolic cost that have been discussed in \cref{subsec:ens_OC_form}. 
Namely, we aim to formulate an adaptive therapy that allows tumor growth to be controlled at the early stage of the therapy, without exceeding the progression threshold.
More precisely, we define the `Off-On'~AT through the following steps, for every $\theta\in\Theta_N$:
\begin{enumerate}
\item At $\tau =0$, set {$u=0$} and keep it constant while $s^\theta(\tau) + r^\theta(\tau)<  1.2\, n_0$ (vacation period).
\item After observing $s^\theta(\tau)+r^\theta(\tau)\geq 1.2\, n_0$, reset {$u=1$} and keep it constant while  $s^\theta(\tau) + r^\theta(\tau)> n_0/2$ (treatment period).
\item After observing $s^\theta(\tau) + r^\theta(\tau) \leq n_0/2$, reset {$u=0$} and go to Step~(1).
\end{enumerate}
The initial `vacation period' allows the reproduction of sensitive cells to promote an increased fight for resources with the resistant clones. However, as soon as the tumor size gets close to the progression threshold, the therapy is started. Finally, the therapy is discontinued when the total population shrinks below the $50\%$ of the initial size $n_0$.
We present the results in \cref{table:Off-On}, and we show the {progression} of the disease for a specific parameter $\theta\in \Theta_N$ in \cref{fig:Off-On}. We report again the information about `On-Off'~AT to facilitate the comparison.

\begin{table}[h]
\caption{`Off-On' Adaptive Therapy}
\label{table:Off-On}

\centering
\begin{tabular}{c c c c c}
\toprule
& $n_0$ & max TTP & min TTP & mean TTP \\
\midrule

`On-Off' AT  & $0.25$ & $575$ days  & $109$ days & $213$ days \\
`Off-On' AT  & $0.25$ & $589$ days  & $107$ days & $217$ days \\
\midrule

`On-Off' AT  & $0.50$ & $764$ days  & $155$ days & $285$ days \\
`Off-On' AT  & $0.50$ & $822$ days  & $167$ days & $306$ days \\
\midrule

`On-Off' AT  & $0.75$ & $1196$ days & $283$ days & $462$ days \\
`Off-On' AT  & $0.75$ & $1500$ days & $400$ days & $589$ days \\
\bottomrule
\end{tabular}

\vspace{5pt}
\footnotetext{Comparison in terms of Time-to-Progression (TTP) between `On-Off'~AT and `Off-On'~AT. The mean TTP is computed by taking the average over the elements of the set $\Theta_N$ (see \cref{eq:Theta_N_simul}).}
\end{table}

\begin{figure}
    \centering
    \includegraphics[width=0.49\linewidth]{ADA1_specif.eps}
    \includegraphics[width=0.49\linewidth]{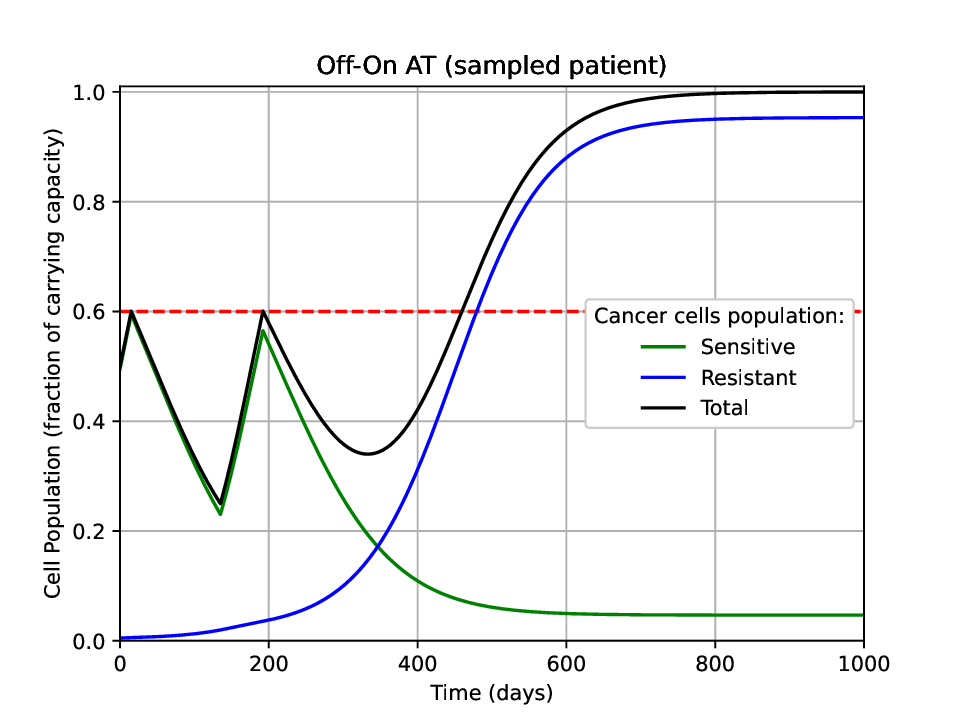}
    \caption{\footnotesize {Dynamics} of the tumor corresponding to $\theta = (1.5,0,0.66,0.01)$ and with initial size $n_0=0.5$ using `On-Off'~AT (left) and `Off-On'~AT (right).
    The dashed horizontal line represents the threshold tumor size related to the condition `cancer in progression'.
    The TTP of `On-Off'~AT and `Off-On'~AT is $424$ days and $459$ days, respectively, with a progression delay of $5$ weeks.}
    \label{fig:Off-On}
\end{figure}

\noindent
The outcomes of `Off-On'~AT seem promising, as it shows an improvement in almost every performance indicator compared to `Off-On'~AT. The only exception is $\min$ TTP for $n_0=0.25$, where, in the worst-case tumor, the progression occurs $2$~days earlier with `Off-On'~AT than when adopting `On-Off'~AT.
{
\begin{remark} \label{rmk:delay_rationale}
In the idealized setting considered here, waiting until the threshold $s^\theta(\tau)+r^\theta(\tau)\geq 1.2\, n_0$  is crossed increases the effect of resource competition exerted by sensitive cells on resistant ones, thereby delaying overall tumor progression, independently of baseline tumor size. 
This threshold should not be interpreted as a direct clinical recommendation: in patients with high baseline tumor burden, early intervention may be required. 
For prostate cancer, our analysis is consistent with low-burden or biochemical recurrence scenarios, where treatment initiation is guided by PSA dynamics (e.g., EMBARK trial criteria \cite{Freedland2023}).
\end{remark}
}

\section*{Conclusions}

From a theoretical standpoint, the main contribution of this work is the development of an ensemble optimal control framework in a general control-affine setting, which is sufficiently flexible to encompass a wide class of cancer models. 
In particular, it accommodates multi-species (e.g., partially resistant or tolerant sub\-populations) tumor dynamics with inter-species interactions, as well as control-dependent transition mechanisms that naturally encode therapy-induced effects.
Moreover, the formulation is compatible with pharmacological refinements that decouple drug dose from drug concentration through additional dynamical equations, thereby moving beyond the assumption of direct proportionality between administration and effect.\\
A natural limitation of the present theoretical framework lies in the restriction to a fixed finite time horizon $T>0$. In this context, it would be particularly relevant to investigate an ensemble \emph{time-optimal} control problem, where the objective is directly related to clinically meaningful quantities such as the `time to progression'. 
This direction appears especially promising in view of applications in long-term disease management. Further extensions include the study of singular arcs (see \cite{aronna2025singular}) and turnpike-type properties in the ensemble setting (see \cite{hernandez2024averaged} for parameter-dependent problems, and \cite{trelat2025turnpike} for a survey), building on recent developments in optimal control theory.

In the computational part, we focused on a two-population model for prostate cancer. We recall that, after observing biochemical recurrence of prostate carcinoma without radiological evidence of metastasis, the choice between immediate or delayed starting of hormone therapy remains an open question to the present day.
Adopting an active surveillance strategy with periodic clinical, instrumental, and laboratory reassessment of the disease remains a viable option. On the one hand, in clinical practice, this strategy is mainly applied in cases of indolent disease with long waiting time before biochemical recurrence after radical treatment ($>18$ months), long time to PSA doubling ($>12$ months), low Gleason score at diagnosis ($6$ or $7$), and PSA levels $<1$~ng/mL post-prostatectomy or $<2$~ng/mL post radiotherapy \cite{preisser2024european,karim2025early}.
On the other hand, in the setting of metastatic disease, immediate initiation of systemic treatment remains a cornerstone of cancer therapy.\\
The results obtained in \cref{subsec:ens_OC_form} for the hyperbolic cost suggest that the active surveillance (already used in clinical practice) can be a nearly-optimal strategy for the long-term management of the disease {within the proposed modeling framework. However, this conclusion should be interpreted in light of the simplifying assumptions that underlie the model in \cref{eq:intro_model_nondim}, which do not explicitly account for more complex biological mechanisms such as resistance development, phenotypic plasticity, or spatial heterogeneity.}\\
{In this part, we have focused on a two-population model calibrated using parameter estimates reported in the literature for prostate carcinoma (non-metastatic and castration-sensitive, m0CSPC) treated with androgen deprivation therapy (ADT). While this setting provides a clinically relevant testbed for the proposed methodology, it should be understood as a simplified instance of the broader class of models covered by the theoretical framework, rather than a complete description of the underlying biological complexity.}

{The `Off-On' Adaptive Therapy proposed in \cref{subsec:Off-On} is motivated by the structure of the optimal controls obtained in the ensemble problem with hyperbolic cost, which suggests a delayed initiation of treatment. 
However, in contrast to the purely optimal ensemble policy, the proposed strategy retains a clinically interpretable rule in which therapy is triggered by tumor burden, after an initial surveillance phase. 
Although the computational study is performed in the context of prostate cancer treated with ADT, mainly due to the availability of parameter estimates in the literature, the Off-On structure is not disease-specific and can in principle be applied to other settings where competition between sensitive and resistant populations is present. 
From a clinical perspective, the strategy aims at improving the `time to progression' compared to standard maximum tolerated dose protocols, and in our simulations it also outperforms the previously proposed On-Off~AT. 
A relevant limitation is that the possibility of safely delaying treatment initiation must ultimately be assessed on a case-by-case basis by clinicians, even if such a delay appears advantageous from the modeling viewpoint.}

\subsection*{Acknowledgements}

Alessandro Scagliotti acknowledges support from the ERC Advanced Grant NEITALG, grant agreement No.~101198055 (P.I.: Prof.~Massimo Fornasier).\\

\bigskip
\begin{center}
  \FundingLogos
  
  \vspace{0.5em}
  \begin{tcolorbox}\centering\small
    Funded by the European Union. Views and opinions expressed are however those of the author(s) only and do not necessarily reflect those of the European Union or the European Research Council Executive Agency. Neither the European Union nor the granting authority can be held responsible for them.
    This project has received funding from the European Research Council (ERC) under the European Union’s Horizon Europe research and innovation programme (grant agreement No.~101198055, project acronym NEITALG).
    
  \end{tcolorbox}
\end{center}


\begin{appendix}
\section{Technical proofs of Section~2}
    \label{sec:appendix_theory}

{
\begin{proof}[Proof of \cref{lem:domain_invariance}]
Let us fix $\theta \in \Theta$ and $D\in \U$, and, to simplify the notations, let us set $X(t)\coloneqq X_D^\theta(t)$ where $X_D^\theta(t) \in \Delta^\theta$ (see \cref{eq:simplex_start_theta}).
First, arguing component by component, we notice that the system is quasi-positive: For $i\neq j$, $A_{i,j},A^I_{i,j}\geq 0$, so that, if $X_i(t)=0$, then $\dot X_i(t)\geq 0$. 
Therefore, it follows that $X_i(t)\geq 0$ for all $t\in[0,T]$.\\
Then, let us study $N(t)\coloneqq \sum_{i=1}^n X_i(t)$. Summing the equations and using \eqref{eq:pop_conservation}, the mutation terms cancel, yielding
\[
\dot N(t)
= \sum_{i=1}^n \left[ r_i\Big(1-\tfrac{N(t)}{K}\Big)\Big(1-2d_i^I\tfrac{D(t)}{D_{\max}}\Big) - d_i^T \right] X_i(t).
\]
If $N(t)=K$, then $\dot N(t) = -\sum_{i=1}^n d_i^T X_i(t)\leq 0$, so  that we deduce that $N(t)\leq K$ for all $t\in[0,T]$.\\
These two arguments show that $X_i(t)\geq 0$ for every $i=1,\ldots,n$ and that $\sum_{i=1}^NX_i(t)\leq K$, i.e., that $X(t)\in \Delta^\theta$ for all $t\in[0,T]$.
\end{proof}}

\begin{proof}[Proof of \cref{lem:diff_traj}]
    In the case of $\theta\in\Theta$ fixed, the result follows directly from \cite[Proposition~2.4]{scag23gf}. 
    However, here the thesis requires \cref{eq:1st_ord_traj} to hold uniformly as $\theta$ varies in $\Theta$.
    To see this, we need to investigate how the dependence on $\theta$ of the fields $F_0^\theta, F_1^\theta$ affects the estimates in the proof of \cite[Proposition~2.4]{scag23gf}.
    From \cite[Proposition~2.3]{scag23gf}, we first observe that, for every $\tau\in[0,T]$,
    \begin{equation} \label{eq:lip_dep_traj}
        |X_{{D+\e\delta}}^\theta(\tau) - X_{{D}}^\theta(\tau)| \leq C_1 \|{\delta}\|_{L^2} \e, \quad C_1\coloneqq 3\sqrt 2 e^{\sqrt 2 L\|{D}\|_{L^2} T},
    \end{equation}
    where $L>0$ is the Lipschitz constant of the vector fields $F_0^\theta,F_1^\theta$. Recalling that $\Theta$ is compact and observing the smooth dependence of $\tilde F_0^\theta, \tilde F_1^\theta$ in $\theta$ (see \cref{eq:aff_ctrl_sys}), we conclude that $F_0^\theta,F_1^\theta$ are uniformly Lipschitz continuous in the state variable, since they are a smooth truncation supported on $B_{3K_{\max}}(0)\subset \R^{n}$ of $\tilde F_0^\theta, \tilde F_1^\theta$. Hence, we conclude that \cref{eq:lip_dep_traj} holds uniformly in $\theta$, which in turn implies that
    \begin{equation*}
      \sup_{\theta\in\Theta}  \e  \left|F_i^\theta\big( X_{{D+\e\delta}}^\theta(\tau) \big) - F_i^\theta\big( X_{{D}}^\theta(\tau) \big)\right| \leq L C_1  \|{\delta}\|_{L^2} \e^2 \qquad i=0,1.
    \end{equation*}
    Finally, we need to show that there exists a modulus of continuity $\delta\colon \R_+ \to \R_+$ such that 
    \begin{equation*}
        \left| 
        F_i^\theta (x_2) - F_i^\theta (x_1) - \frac{\partial F_i^\theta(x_1)}{\partial x}(x_1-x_2) 
        \right| \leq \delta(|x_2-x_1|) |x_2-x_1| \qquad i=0,1
    \end{equation*}
    for every $\theta\in \Theta$ and for every $x_1,x_2 \in \R^{n}$. However, this follows again from the smooth dependence of $\tilde F_0^\theta, \tilde F_1^\theta$ in $\theta$ and from the fact that $F_i^\theta(x) = \rho(x) \tilde F_i^\theta(x)$ for $i=0,1$, with $\rho\colon \R^{n}\to \R$ smooth and compactly supported cut-off function.
    Having observed that, we deduce that the estimates done in the proof of \cite[Proposition~2.4]{scag23gf} hold uniformly in $\theta$, and we deduce the thesis.
\end{proof}

\begin{proof}[Proof of \cref{lem:tangent_domain}]
    Denoting for brevity with $C_{D}$ the right-hand side of \cref{eq:char_tangent}, we first show that $T({D},\U)\subseteq C_{D}$.
    Let us consider a sequence $(\e_n)_{n\geq 1}$ such that $\e_n\to 0$ as $n\to\infty$, and let us consider $v_{\e_n} =\frac{{D}'_n- {D}}{\e_n} \in \frac{\U- {D}}{\e_n}$, with ${D}'_n\in\U$. Let us introduce
    \begin{equation} \label{eq:instant_0_1}
        \begin{split}
            A_0\coloneqq \{t\in[0,T] : {D}(t)=0\}, \\
            A_1\coloneqq \{t\in[0,T] : {D}(t)={D_{\max}}\}.
        \end{split}
    \end{equation}
    Since $0\leq {D}'_n\leq {D_{\max}}$ a.e.~and for every $n\geq 1$, it turns out that
    \begin{equation*}
        \begin{split}
            0\leq v_{\e_n}(t)\leq \frac{{D_{\max}}}{\e_n} \quad \mbox{a.e. on } A_0,\\
            -\frac{{D_{\max}}}{\e_n} \leq v_{\e_n}(t)\leq  0 \quad \mbox{a.e. on } A_1,
        \end{split}
    \end{equation*}
    which, in particular, implies that $v_{\e_n}(t)\geq 0$ and $v_{\e_n}(t)\leq 0$ a.e.~on $A_0,A_1$, respectively.
    Moreover, let $v \in L^2([0,T],\R)$ be a $L^2$-strong cluster point of the sequence $(v_{\e_n})_{n\geq 1}$. Therefore, there exists a subsequence of $(v_{\e_n})_{n\geq 1}$ that is converging to $v$ at a.e.~$t \in [0,T]$, and we deduce that $v \in C_{D}$.\\
    We now address the inclusion $C_{D} \subseteq T({D},\U)$. Let us fix $v\in C_{D}$ and $\delta>0$. We aim at constructing $\bar \e>0$ and ${D}'_{\bar \e}\in \U$ such that $\| v- \frac1{\bar \e}({D}'_{\bar \e} - {D} )\|_{L^2}\leq \delta$.
    To do that, we first choose $M>0$ such that $\| v- v_M\|_{L^2}\leq \frac\delta2$, where $v_M\in L^2([0,T],\R)$ is defined as the truncation of $v$, i.e., $v_M(t)\coloneqq \min\left( \max\big( v(t), -M \big),M \right)$ for a.e.~$t\in [0,T]$.
    Then, let us introduce $A_b\coloneqq [0,T]\setminus (A_0 \cup A_1)$, i.e., $A_b= \{ t\in[0,T]: 0<{D}(t)<{D_{\max}} \}$.
    Moreover, we define $A_b^k \coloneqq \left\{ t\in[0,T]: \min\big( {D}(t),{D_{\max}}-{D}(t) \big) \geq 1/k \right\}$ for $k\geq 2$, and we observe that 
    \begin{equation*}
        A_b^k\subseteq A_b^{k+1} \quad \forall k\geq 2, \quad \mbox{and } \quad
        A_b= \bigcup_{k=2}^\infty
        A_b^k.
    \end{equation*}
    This implies that there exists $\bar k\geq 2$ such that 
    \begin{equation*}
        \left\| v_M - v_M \mathds{1}_{A_0\cup A_1\cup A_b^{\bar k}}  \right\|_{L^2}\leq \frac\delta2,    
    \end{equation*}
    where $\mathds{1}_B\colon [0,T]\to \{0,1 \}$ denotes the function that is equal to $1$ on the Borel set $B$, and $0$ elsewhere. Setting $A^\delta\coloneqq A_0\cup A_1\cup A_b^{\bar k}$ for brevity, using the triangular inequality, we notice that $\left\| v - v_M \mathds{1}_{A^\delta}  \right\|_{L^2}\leq \delta$.
    We are left to show that, setting $\bar \e = 1/(\bar k  M)$, the function ${D}'_{\bar \e}\coloneqq {D} + \bar \e v_M\mathds{1}_{A^\delta}$ belongs to $\U$, i.e., $0\leq {D}'_{\bar \e} \leq {D_{\max}}$ a.e.
    We observe that:
    \begin{itemize}
        \item If $t \in A_0$, then by definition of $C_{D}$ we have that $v(t)\geq 0$ and $M\geq v_M(t) \geq 0$, while $\mathds{1}_{A^\delta}(t)=1$.
        Hence, ${D}'_{\bar \e}(t) = \bar \e v_M(t) \in [0, {D_{\max}}/\bar k]$, and in particular $0\leq {D}'_{\bar \e}(t) \leq {D_{\max}}$.
        \item If $t \in A_1$, then by definition of $C_{D}$ we have that $v(t)\leq 0$ and $-M \leq v_M(t) \leq 0$, while $\mathds{1}_{A^\delta}(t)=1$.
        Hence, ${D}'_{\bar \e}(t) = {D_{\max}} + \bar \e v_M(t) \in [{D_{\max}}- {D_{\max}}/\bar k, {D_{\max}}]$, and in particular $0\leq {D}'_{\bar \e}(t) \leq {D_{\max}}$.
        \item If $t \in A_b^{\bar k}$, we have that $-M \leq v_M(t) \leq M$ and ${D}(t) \in [1/\bar k, {D_{\max}}-1/\bar k]$, while $\mathds{1}_{A^\delta}(t)=1$.
        Hence, ${D}'_{\bar \e}(t) = {D}(t) + \bar \e v_M(t)\in [0,{D_{\max}}]$.
        \item Finally, if $t \in A_b\setminus A_b^{\bar k}$, we observe that $\mathds{1}_{A^\delta}(t)=0$, and therefore we have ${D}'_{\bar \e}(t) = {D}(t)\in [0,{D_{\max}}]$ since ${D}\in \U$.
    \end{itemize}
    Since this shows the inclusion $C_{D} \subseteq T({D},\U)$, the proof is complete. 
\end{proof}
    
\end{appendix}

\bibliographystyle{abbrv}
\bibliography{biblio}

\end{document}